\documentclass[11pt]{amsart}

\usepackage{xcolor}
\usepackage{amssymb,latexsym,amsmath,extarrows}
\numberwithin{equation}{section}
\usepackage{url}
\usepackage{mathrsfs}
\usepackage{bbm}
\usepackage{amsmath}
\usepackage{mathabx}
\usepackage{graphicx}
\usepackage{pgfplots}

\def\XXint#1#2#3{{\setbox0=\hbox{$#1{#2#3}{\int}$ }
\vcenter{\hbox{$#2#3$ }}\kern-.6\wd0}}

\newtheorem{theorem}{Theorem}[section]
\newtheorem{corollary}{Corollary}[theorem]
\newtheorem{lemma}[theorem]{Lemma}
\newtheorem{conjecture}[theorem]{Conjecture}
\theoremstyle{definition}
\newtheorem{definition}{Definition}[section]
\newtheorem{remark}[theorem]{Remark}

\theoremstyle{proposition}
\newtheorem{proposition}{Proposition}[section]

\begin{document}

\title[Improvement on the Circle and Divisor Problems]
{An improvement on Gauss's Circle Problem and Dirichlet's Divisor Problem}
  
\author{Xiaochun Li}

\address{Xiaochun Li\\
Department of Mathematics\\
University of Illinois at Urbana-Champaign\\
Urbana, IL, 61801, USA}

\email{xcli@illinois.edu}

\author{Xuerui Yang}

\address{ Xuerui Yang\\
Department of Mathematics\\
University of Illinois at Urbana-Champaign\\
Urbana, IL, 61801, USA}

\email{xueruiy3@illinois.edu}

\date{}

\begin{abstract}
Using Bombieri-Iwaniec method, we establish an improvement on both Gauss's Circle Problem and Dirichlet's Divisor Problem. More precisely, we derive a new estimate for the first spacing problem and combine it with Huxley's results on the second spacing problem.
\end{abstract}

\maketitle

\section{Introduction}  

In the realm of number theory, certain fundamental problems have persisted as intricate challenges, captivating the minds of mathematicians for generations. Two such enigmas are Gauss's Circle Problem and the Dirichlet's Divisor Problem, which seek to uncover the distribution of integral points within circles and the distribution of divisors of integers, respectively. 


Let 
\begin{equation}  \label{circle problem}
R(X)=\sum_{m^2+n^2\le X} 1-\pi X    
\end{equation}
and 
\begin{equation}  \label{divisor problem}
\Delta(X)=\sum_{1\le n\le X} d(n) - X\log X -(2\gamma-1)X.
\end{equation}
Here $d(\cdot)$ is the divisor function
\[
d(n)=\sum_{s\mid n} 1, 
\]
and $\gamma=0.5772\dots$ is Euler's constant. Gauss's Circle Problem and the Dirichlet's Divisor Problem seek for upper bounds of the following form:
\[
R(X)\lesssim_\epsilon X^{\theta+\epsilon}, 
\]
and 
\[
\Delta(X)\lesssim_\epsilon X^{\theta+\epsilon},
\]
with $\theta$ as small as possible. 

For a detailed account of the history and background of these two problems, we refer the readers to the comprehensive survey paper by B. C. Berndt, S. Kim, and A. Zaharescu \cite{Bruce}.  For both problems, attempts on the opposite direction
were made in a series of papers \cite{Hardy1},\cite{Hardy2} and \cite{Hardy3} by G. H. Hardy, who proved that both $R(X)$ and $\Delta(X)$ are 
\begin{equation}\label{Ome-1}
\Omega ((X\log X)^{\frac{1}{4}}),
\end{equation}
where $f=\Omega(g)$ means that for any constant $C>0$, $|f|\ge C |g|$ infinitely often in the limit process. Since then, no one has been able to enlarge the exponent $\frac{1}{4}$ on $X$ in  (\ref{Ome-1}).  That naturally gives rise to the following conjecture:
\begin{conjecture}   \label{conj-1}
In Gauss's Circle Problem and Dirichlet's Divisor Problem, 
\[
\theta=\frac{1}{4}. 
\]
\end{conjecture}
More recently, the pursuit of these two famous problems has been invigorated by the Bombieri-Iwaniec method, which was initiated in \cite{BombieriIwaniec} to study the pointwise bound on $|\zeta(\frac{1}{2}+it)|$, where $\zeta(s)$ is the Riemann zeta function. Then it was adapted by H. Iwaniec and C. J. Mozzochi \cite{IwaniecMozzochi} to study Gauss's Circle Problem and Dirichlet's Divisor Problem. 
Later on, M. N. Huxley refined and generalized the method in a series of papers \cite{HuxleyZeta1},\cite{HuxleyZeta4},\cite{HuxleyZeta5},\cite{HuxleyCircle1},\cite{Huxley03} and \cite{HuxleyRC}, addressing the pointwise estimate of $|\zeta(\frac{1}{2}+it)|$ and the Circle and Divisor Problems. Before our work, the best-known upper bound for $\theta$ was established by Huxley \cite{Huxley03} at $\theta=\frac{131}{416}\approx 0.3149\dots$. 
The current record for an upper bound on $|\zeta(\frac{1}{2}+it)|$ was established by J. Bourgain \cite{BourgainZeta}, who proved that 
\[
\Big|\zeta(\frac{1}{2}+it)\Big|\lesssim_\epsilon |t|^{\frac{13}{84}+\epsilon}.
\]
As a novel observation in \cite{BourgainZeta}, it was shown that the decoupling theory can be used to handle certain mean values of
exponential sums arising in the pointwise estimates of the zeta function. 
Using this insight, Bourgain and N. Watt in \cite{BourgainWatt1st} were able to improve bounds for the mean square of $\big|\zeta(\frac{1}{2}+it)\big|$ on short intervals. \\

Combining those aforementioned methods, in this paper, we are able to strengthen Huxley's result as follows. 

\begin{theorem}  \label{theorem in introduction}
\[
R(X)=O_\epsilon (X^{\theta^*+\epsilon}) \quad \text{ and} \quad \Delta(X)= O_\epsilon (X^{\theta^*+\epsilon})
\]
for all $\epsilon>0$, where 
\[
\theta^*= 0.314483\cdots
\]
is defined below.  
\begin{definition}  \label{theta} 
$\theta^*= 0.3144831759741\cdots$ is defined in such a way that $-\theta^*$ is unique solution to the equation
\begin{equation}  \label{definition of theta}
-\frac{8}{25}x-\frac{1}{200}\Big(\sqrt{2(1-14x)}-5\sqrt{-1-8x}\Big)^2+\frac{51}{200}=-x    
\end{equation}
on the interval $[-0.35,-0.3]$.

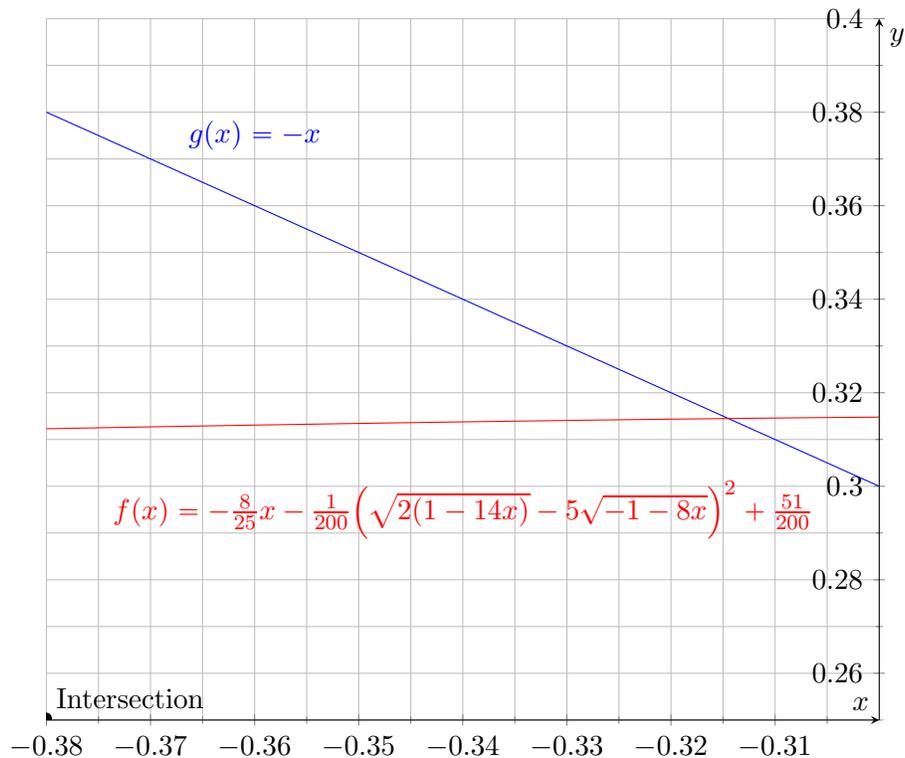
\begin{tikzpicture}
\begin{axis}[
    xlabel=$x$,
    ylabel=$y$,
    width=\textwidth, 
    xmin=-0.38, xmax=-0.3, 
    ymin=0.25, ymax=0.4, 
    axis lines=middle,
    grid=both,
    minor tick num=1,
    mark size=3pt,
    domain=-0.38:-0.3, 
    samples=100,
    legend style={at={(0.05,0.95)},anchor=north west},
]
\addplot[blue,smooth] {-x};
\addplot[red,smooth] {-(8/25)*x - (1/200)*((sqrt(2*(1-14*x))-5*sqrt(-1-8*x))^2) + (51/200)};

\coordinate (intersection) at (-0.35, 0.35);
\draw[fill=black] (intersection) circle (2pt) node[above right] {Intersection};

\node[blue] at (axis cs: -0.36, 0.375) {$g(x)=-x$};
\node[red] at (axis cs: -0.34, 0.295) {$f(x)=-\frac{8}{25}x-\frac{1}{200}\Big(\sqrt{2(1-14x)}-5\sqrt{-1-8x}\Big)^2+\frac{51}{200}$};

\end{axis}
\end{tikzpicture}
\end{definition}
\end{theorem}

Let us outline some of ideas used in our proof.  Starting with the Bombieri-Iwaniec method, we dissect a long exponential sum $S$
such as 
\begin{equation}\label{zeta1}
\sum_{m\sim M} e(T\log m) \,,\quad \quad \text{where } M\ll T,    
\end{equation}
and 
\begin{equation} \label{c-d}
\sum_{h\sim H}\sum_{m\sim M} e\Big(\frac{Th}{m} \Big) \,,\quad \text{where }  H\ll M\ll T,    
\end{equation}
into an addition of short sums $\sum_j S_j$. A sum of the form (\ref{zeta1}) arises in the study of 
$\Big|\zeta(\frac{1}{2}+it)\Big|$, and (\ref{c-d}) is the typical exponential sum encountered in both the Circle and Divisor Problems. 
For each short sum $S_j$, a Taylor expansion is employed to replace intractable phase functions, like $\log m$ and $\frac{T}{m}$, 
by cubic or quadratic polynomials in $m$. In other words, each short exponential sum can be transformed into a  standard Weyl sum, in which the coefficients are very close to
nice rational numbers via a use of the Dirichlet approximation theorem, a standard step in Hardy-Littlewood's circle method. 
For each short Weyl sum, we apply the Poisson summation formula, where the power of the Fourier transform is taken into account. Because of the rational approximations of the coefficients, the Poisson summation can be carried out in a neater way, and the phase functions in the consequent exponential sum will be more structured.  In Gauss's Circle Problem and Dirichlet's Divisor Problem, 
usually only the linear term coefficient is approximated by a rational number, so that the way to decompose $S$ into short sums $S_j$'s is uniquely determined by the rational approximation.  Those short sums $S_j$'s can be classified 
according to the size of the denominator $r$ of the rational approximation $a/r$ to the linear coefficient, corresponding to
major arcs and minor arcs in the Hardy-Littlewood circle method. The Poisson summation and the rational approximation lead our problem to the following exponential sum 
\begin{equation}\label{inner}
\sum_{\frac{a}{r}} \Big| \sum_{(k,l)} e\Big(\Vec{x}_{\frac{a}{r}}\cdot \Vec{y}_{(k,l)}\Big) \Big|,
\end{equation}
where $\Vec{x}_{\frac{a}{r}}$ is a vector with entries related to the rational number $\frac{a}{r}$ from the minor arcs, and
\[
\Vec{y}_{(k,l)}=\Big(l,kl,l\sqrt{k},\frac{l}{\sqrt{k}}\Big).
\]
To handle (\ref{inner}), the standard way is to apply the double large sieve inequality as Bombieri and Iwaniec did. 
After that, we run into Huxley-type first and second spacing problems. The first spacing problem involves estimating on the mean value of exponential sums associated to the vector $\Big(l,kl,l\sqrt{k},\frac{l}{\sqrt{k}}\Big)$.  When the exponent is an even integer, the first spacing problem is equivalent to counting solutions of a certain Diophantine system. However, it was observed by Bourgain  that it is indeed a decoupling problem, deeply connected with modern Fourier analysis.  
This gives a new perspective to the first spacing problem. 
The second spacing problem, roughly speaking,  is related to the distribution of integer points near a $C^3$-curve. Huxley gave some nice estimations on it (see \cite{Huxley03}).   \\

Our endeavor to the first spacing problem is fueled by the recent advancements in the field of Decoupling Theory, particularly the progress on small cap decoupling for the cone made by Guth and Maldague \cite{GuthMaldague}. This will be illustrated in Section \ref{first spacing problem}. Combining our efforts in the first spacing problem with Huxley's second spacing estimates  in \cite{Huxley1996}, \cite{Huxley03}, we can achieve an improvement. This will be the subject of Section \ref{bounds on exponential sum}. In Section \ref{final argument}, we establish Theorem \ref{theorem in introduction} using results from previous sections.\\

It seems that $\theta=\frac{5}{16}=0.3125$ is an unbeatable barrier of the existing methods.  
It will be exciting if the obstacle $\theta=0.3125$ can be overcome.  Conjecture \ref{conj-1} is extremely challenging. Although it is effective to provide non-trivial decay estimates, the Bombieri-Iwaniec method creates an issue in the beginning when the long exponential sum is broken into short sums. The orthogonality between those short sums is retrieved partially through the large sieve method. However, we do not understand how to quantitatively capture the complete cancellation from those short sums. This makes it impossible to reach $\theta=1/4$  as conjectured.  Some new ideas must be needed if one wants to resolve the conjecture. \\

{\bf Acknowledgement} The authors thank Bruce C. Berndt for a careful reading of an earlier version of this paper and for some helpful suggestions that have increased readability of the paper.  

The first author is supported by a Simons fellowship 2019-2020 and Simons collaboration grants.

\section{Notations} \label{notations}

$A\lesssim B$ (or $A=O(B)$) means $A \le C B$,  where $C$ is some positive constant, and $A\lesssim_\epsilon B$ (or $A=O_\epsilon(B)$) indicates that the implicit constant may depend on the subscript $\epsilon$.  

$A\sim B$ denotes that we have both $A\lesssim B$ and $B\lesssim A$. 

$A\asymp B$ represents that we have $B\le A\le 2B$. 

In this paper, $\ll,\gg$ are not Vinogradov's notation.  Instead, $A\ll B$ (or $A\gg B$) suggests that $A$ is much smaller (or larger) than $B$. 

$\|x\|$ is the distance between $x$ and the nearest integer to $x$. 

For $e(x):=e^{2\pi i x}$, the Fourier transform and the inverse Fourier transform are defined as 
\[
\begin{split}
\widehat{f}(\xi)& =\int f(x)e(-x\xi) dx, 
\\ \widecheck {g}(x) &= \int g(\xi) e(x\xi) d\xi. 
\end{split}
\]

$\frac{a}{r}$ is always assumed to be a reduced fraction in this paper.

Lastly, $\text{diam}(a_1,a_2,\dots, a_k)\le C$ indicates that $|a_i-a_j|\le C$ for any pair $(i,j)$, where $1\le i, j\le k$.   

\section{Improvement on the first spacing problem} \label{first spacing problem}

\subsection{The first spacing problem} Bourgain and Watt first noticed that the double large sieve inequality can be derived using Hölder's inequality (\cite{BourgainWatt1st} Section 5). In this way, there is a freedom to choose the exponent of the norm in the first spacing problem. We follow this observation and work with $q$ a little bit larger than $4$. More specifically, the first spacing problem is asking for a nice upper bound for the following norm:
\begin{equation} \label{original form}
G_q=\Big\|\sum_{k\sim K}\sum_{l\sim L} a_{kl}e(lx_1+klx_2+l\sqrt{k}x_3)\Big\|_{L^q_{\#}\Big[|x_1|\le 1,|x_2|\le 1,|x_3|\le \frac{1}{\eta L\sqrt{K}}\Big]},   
\end{equation}
where $a_{kl}$ are arbitrary coefficients such that $|a_{kl}|\le 1$, $q\ge 4$, and the subscript $\#$ denotes the averaged norm:
\[
\|f\|_{L^p_{\#}(B)}=\bigg(\frac{1}{|B|}\int_{B}|f|^p\bigg)^{\frac{1}{p}}.
\]
In \eqref{original form}, the parameters $K,L$ are integers, $\eta>0$, and they satisfy
\begin{equation}  \label{relations between L,K,eta}
1\le L<K\le \frac{1}{\eta}\le KL.    
\end{equation}
Throughout this paper, we only consider $q$ with the following constraints, 
\begin{equation}   \label{range of q}
 4\le q\le 4.5.    
\end{equation}

\quad

Analogous to the interpretation of the norm in Vinogradov's Mean Value Theorem, if we set $a_{kl}=1$ for all $k,l$ in \eqref{original form}, and let $q=2n$ be an even positive integer, then $G^q_q$ is equal to the number of integer solutions of the following system:

\begin{align}
l_1+\dots +l_n &=l_{n+1}+\dots +l_{2n},     \label{condition 1} 
\\  k_1l_1+\dots +k_n l_n &=k_{n+1}l_{n+1}+\dots +k_{2n} l_{2n}, \label{condition 2}
\\ l_1\sqrt{k_1} +\dots + l_n\sqrt{k_n}&=l_{n+1}\sqrt{k_{n+1}}+\dots + l_{2n}\sqrt{k_{2n}}+O(\eta L\sqrt{K}),   \label{condition 3}
\\ k_i \sim K, l_i\sim L, & \quad \forall i=1,...,2n.  \label{condition 4}
\end{align}
For an even integer $q$, if we let $a_{kl}$ vary, then $G_q$ is maximum when all $a_{kl}=1$. Thus it suffices to look at this special case. 

There are many trivial solutions to the above system. If we set $k_{i+n}=k_i$ and $l_{i+n}=l_i$ for all $i=1,...,n$, the system always holds, no matter what values $k_1,\dots,k_n,l_1,\dots,l_n$ we take. So the number of solutions is $\gtrsim K^n L^n$, which implies 
\begin{equation}  \label{lower bound}
G_q\gtrsim K^{\frac{1}{2}} L^{\frac{1}{2}},  
\end{equation}
when $q=2n\ge 2$. Since we have normalized the measure space in \eqref{original form}, by Hölder's inequality, \eqref{lower bound} should be true for all real $q\ge 2$.  

\subsection{Upper bound on the first spacing problem}

When Huxley worked on the first spacing problem, he considered the case $q=4$, and he derived essentially a sharp estimate (\cite{Huxley1996} Theorem. 13.2.4):
\[
G_4 \lesssim_{\epsilon} K^\epsilon (KL)^{\frac{1}{2}}. 
\]
As stated before, we assume that $q$ is a little bit larger than $4$. We then obtain the following proposition: 
\begin{proposition}[Upper bound on $G_q$]  \label{main proposition statment} 
For $q\ge 4$, if $\eta$ satisfies   
\begin{equation}  \label{main proposition assumption}
\Big(\frac{L}{K}\Big)^{\frac{q-2}{q-4}}\le \eta,    
\end{equation}
then
\begin{equation}    \label{main proposition}
G_q \lesssim_\epsilon \eta^{-\epsilon} \eta^{\frac{q-4}{q(q-2)}} (KL)^{1-\frac{2}{q}}\Big(1+\eta^{\frac{2}{q-2}}K\Big)^{\frac{1}{q}}.    
\end{equation}
\end{proposition}
\begin{remark}
1) In this proposition, the upper bound on $G_q$ is uniform for all choices of $a_{kl}$, as long as we assume that $|a_{kl}|\le 1$. 

2) We note that \eqref{main proposition assumption} is satisfied when $q$ is sufficiently close to $4$, since $\frac{L}{K}< 1$.

3) In our ultimate application of this proposition in Section \ref{final argument}, the parameters satisfy
\[
\eta^{\frac{2}{q-2}}K \lesssim 1,
\]
which implies that the last term in \eqref{main proposition} is $1$. This is convenient for computation.
\end{remark}

\subsection{Proof of Proposition \ref{main proposition statment}}

To prove the proposition, we use the small cap decoupling theorem for a truncated cone proved by Guth and Maldague \cite{GuthMaldague}. Before stating and applying their theorem, we perform some transformations on \eqref{original form}. In this way, the exponential sum can be viewed as the inverse Fourier transform of some Schwartz function defined on a neighborhood of the truncated cone, which fits in the setting of the decoupling theorem.  \\

Let us  start with some definitions.
\begin{definition}[The truncated cone and its neighborhood]  \label{truncated cone}
Let
\begin{equation}   \label{cone definition}
\mathcal{C}=\{(\xi_1,\xi_2,\xi_3):\xi_1^2+\xi_2^2=\xi_3^2, \quad \xi_3\sim 1\}    
\end{equation} 
be the truncated cone, and let $\mathcal{N}_\eta(\mathcal{C})$ be a $\eta$-neighborhood of $\mathcal{C}$ in $\mathbb R^3$.
\end{definition}

We introduce a notation to denote rectangular box centered at the origin. 

\begin{definition} \label{box centered at origin} Define
\[
B(A_1,A_2,A_3) :=\{(x_1,x_2,x_3)\in \mathbb R^3: |x_1|\le A_1,|x_2|\le A_2, |x_3|\le A_3\},
\]    
where $A_1, A_2, A_3>0$.
\end{definition}

By the definition of $G_q$ \eqref{original form} and  periodicity in the first variable, we can extend the range of the variable $x_1$ and then write $G_q$ as 
\begin{equation}  \label{first by periodicity}
G_q=\Big\|\sum_{k\sim K}\sum_{l\sim L} a_{kl}e(lx_1+klx_2+l\sqrt{k}x_3)\Big\|_{L^q_{\#}\Big(B(K,1,\frac{1}{\eta L\sqrt{K}})\Big)}.    
\end{equation}
Then we normalize the variables $k,l$ by rescaling of variables $x_1,x_2,x_3$ in  (\ref{first by periodicity}) so that  
\begin{equation}  \label{variance 1}
G_q=\Big\|\sum_{k\sim K}\sum_{l\sim L} a_{kl} e\Big(\frac{l}{L}x_1+\frac{kl}{KL}x_2+\frac{l\sqrt{k}}{L\sqrt{K}}x_3 \Big)\Big\|_{L^q_{\#}\Big(B(KL,KL,\frac{1}{\eta})\Big)}.
\end{equation}

\quad

It will be convenient to make a change of variable. Hence we let 
\[
s=\frac{l}{L}\sim 1, \quad t=\frac{k}{K}\sim 1.
\]
Then 
\[\Big(\frac{l}{L},\frac{lk}{LK},\frac{l\sqrt{k}}{L\sqrt{L}}\Big)=(s,st,s\sqrt{t})\,.\]
Performing the following affine transformation, 
\begin{equation}\label{trans}
\xi_1=s\sqrt{t}, \quad \xi_2=\frac{t-1}{2}s,\quad \xi_3=\frac{t+1}{2}s\,,
\end{equation}
we see that  $(\xi_1, \xi_2, \xi_3)$ lies on the cone $\mathcal{C}$ \eqref{cone definition}, because 
\begin{equation}   \label{reason for lying on the cone}
\Big(\frac{st+s}{2}\Big)^2 -\Big(\frac{st-s}{2}\Big)^2=s^2 t=(s\sqrt{t})^2.    
\end{equation}

Replacing $s$ by $l/L$ and $t$ by $k/K$,  we know that the relations between $(k,l)$ and $\Vec{\xi}$ are given by 
\begin{equation}  \label{change of variables 1}
\left\{\begin{array}{rcl}
\xi_1 &= & \frac{l}{L}\frac{\sqrt{k}}{\sqrt{K}} \,\,\,\, \sim 1,
\\ \xi_2 &= &\frac{l}{L} \frac{k/K-1}{2}, 
\\ \xi_3 &= & \frac{l}{L} \frac{k/K+1}{2}  \sim 1.
\end{array}\right.    
\end{equation}

Define
\[
\Gamma=\big\{\Vec{\xi}: \Vec{\xi} \text{ is given by } \eqref{change of variables 1} \text{ for some }(k,l) \text{ s.t. }k\sim K, l\sim L \big\}.
\]
We note that the vectors $\Vec{\xi}\in \Gamma$ are $\frac{1}{K}$-separated in the circular direction and $\frac{1}{L}$-separated in the null direction. Also, it is easy to see that the map \eqref{change of variables 1} is bijective from $\{(k,l):k\sim K, l\sim L\}$ to $\Gamma$, so $a_{\Vec{\xi}}=a_{kl}$ can be defined correspondingly.\\

By \eqref{change of variables 1}, \eqref{variance 1} becomes 
\[
\begin{split}
&\Big \|\sum_{\Vec{\xi}\in \Gamma}a_{\Vec{\xi}}\, e\Big((\xi_3-\xi_2)x_1+(\xi_3+\xi_2)x_2+\xi_1x_3\Big)\Big\|_{L^q_{\#}\Big(B(KL,KL,\frac{1}{\eta})\Big)}    
\\ =& \Big\|\sum_{\Vec{\xi}\in \Gamma}a_{\Vec{\xi}} \,e\Big(\xi_1 x_3 +\xi_2(x_2-x_1)+\xi_3(x_1+x_2)\Big)\Big\|_{L^q_{\#}\Big(B(KL,KL,\frac{1}{\eta})\Big)}.
\end{split}
\]
Performing another change of variables, say, 
\begin{equation}    \label{change of variables 2}
\left\{\begin{array}{rcl}
y_1 &=& x_3
\\ y_2 &= & x_2-x_1
\\ y_3 &=&  x_2+x_1\,,
\end{array}    \right.
\end{equation} 
we find that 
\begin{equation} \label{variance 2} G_q =\Big\|\sum_{\Vec{\xi}\in \Gamma}a_{\Vec{\xi}}\,e\Big(\xi_1 y_1+\xi_2 y_2+\sqrt{\xi_1^2+\xi_2^2}y_3\Big)\Big\|_{L^q_{\#}\Big( B(\frac{1}{\eta},KL,KL) \Big)}. 
\end{equation}
Now we are able to view the exponential sum in \eqref{variance 2} as a function whose Fourier transform is supported on $\mathcal N_\eta(\mathcal C)$, a $\eta$-neighborhood of the cone $\mathcal C$, since $\eta\ge \frac{1}{KL}$.  To see why this is true, we let $\psi$ be a smooth bump function supported in a neighborhood of $[-1,1]$ and taking the value $1$ on $[-1,1]$, and define
\[
\begin{split}
F(z_1,z_2,z_3)  = \frac{K^2L^2}{\eta} \sum_{\Vec{\xi}\in \Gamma}  a_{\Vec{\xi}} \,  \widehat{\psi}\Big(\frac{z_1-\xi_1}{\eta}\Big) \widehat{\psi}\Big(\frac{z_2-\xi_2}{\frac{1}{KL}} \Big)  \widehat{\psi}\Big(\frac{z_3-\sqrt{\xi_1^2+\xi_2^2}}{\frac{1}{KL}}\Big).  
\end{split}
\]
Then its inverse Fourier transform is 
\[
\widecheck{F}(y_1\!,y_2,\!y_3)\!= \!\!\sum_{\Vec{\xi}\in \Gamma}\!a_{\Vec{\xi}}\,e\Big(\!\xi_1 y_1\!+\xi_2 y_2\!+\sqrt{\xi_1^2+\xi_2^2}y_3\!\Big) \psi \Big( \frac{y_1}{\frac{1}{\eta}}\Big) \psi \Big(\frac{y_2}{KL}\Big) \psi\Big(\frac{y_3}{KL}\Big),    
\]
which is the sum on the right side of  \eqref{variance 2} with a Schwartz tail. From now on, by an abuse of notation, we change the dummy variables $(y_1,y_2,y_3)$ to $(x_1,x_2,x_3)$. \\

At this step, we can apply a small cap decoupling theorem of Guth and Maldague to \eqref{variance 2}. Before stating their result, we introduce some concepts first. A generic plate $\sigma$ of dimensions $\eta^{\beta_2}\times \eta^{\beta_1}\times \eta$ is a rectangular box in a $\eta$-neighborhood  of $\mathcal C$ such that $\eta^{\beta_2}$ is the length in the null direction, $\eta^{\beta_1}$ is the length in the circular direction, and $\eta$ is the thickness of the plate. $\mathcal N_\eta(\mathcal C)$ can be covered by essentially pairwise disjoint generic plates $\sigma$ of dimensions $\eta^{\beta_2}\times \eta^{\beta_1}\times \eta$,
where $\beta_2\in[0,1]$, $\beta_1\in [\frac{1}{2},1]$.  By the essential disjointness, we mean that those plates may have finite overlaps but can be divided into finitely many sets, each of which contains disjoint plates. Thus we can view the collection 
of those essentially disjoint generic plates as a partition of $\mathcal N_\eta(\mathcal C)$. Given a Schwartz function $f$, we define $f_\sigma$ by setting its Fourier transform 
\[
\widehat{f_\sigma}=\widehat{f} \cdot \chi_\sigma,  
\]
where $\chi_\sigma$ is the characteristic function of $\sigma$. We now are ready to state Guth and Maldague's theorem. 

\begin{theorem}[\cite{GuthMaldague}, Thm 3]\label{small cap decoupling}
Let $\beta_1\in [\frac{1}{2},1]$ and $\beta_2\in[0,1]$. For $q\ge 2$ and any Schwartz function $f:\mathbb R^3\to \mathbb C$ with Fourier transform supported in $\mathcal{N}_\eta(\mathcal{C})$, we have 
\begin{equation}\label{sm-cap}
\int_{\mathbb R^3} |f|^q \lesssim_\epsilon \eta^{-\epsilon}  D^q_{\beta_1, \beta_2, q}   \sum_{\sigma} \|f_\sigma\|_{L^q(\mathbb R^3)}^q \,,
\end{equation}
where the decoupling constant $D_{\beta_1, \beta_2, q}$ is given by
\[
D_{\beta_1,\beta_2,q}=\eta^{-(\beta_1+\beta_2)(\frac{1}{2}-\frac{1}{q})}+\eta^{-(\beta_1+\beta_2)(1-\frac{2}{q})+\frac{1}{q}}+\eta^{-(\beta_1+\beta_2-\frac{1}{2})(1-\frac{2}{q})}   \,.
\]
\end{theorem}

Applying Theorem \ref{small cap decoupling} to the right side of  \eqref{variance 2}, we get the following lemma immediately.

\begin{lemma}
Let $\beta_1\in [\frac{1}{2},1]$ and $\beta_2\in[0,1]$. For $q\ge 4$,
\begin{equation}  \label{application of small cap decoupling} 
G_q \lesssim_\epsilon \eta^{-\epsilon} D_{\beta_1,\beta_2,q}\Big(\sum_{\sigma}\Big\|\sum_{\Vec{\xi}\in \sigma} F_{\Vec{\xi}} \,(x_1,x_2,x_3)\Big\|^q_{L^q_{\#}\Big(B(\frac{1}{\eta},KL,KL)\Big)}\Big)^\frac{1}{q},   
\end{equation}
where 
\[
F_{\Vec{\xi}}\,(x_1,x_2,x_3)=a_{\Vec{\xi}}\,e\Big(\xi_1 x_1+\xi_2 x_2+\sqrt{\xi_1^2+\xi_2^2}x_3\Big).
\]
We note that $\Vec{\xi}\in \sigma$ in the right side of (\ref{application of small cap decoupling}) can be replaced by
 $\Vec{\xi}\in (\Gamma\cap \sigma)$ because only  those $ \Vec{\xi}\in\Gamma$ makes contributions. 
\end{lemma}

We can simplify the decoupling constant $D_{\beta_1, \beta_2, q}$ when $q\geq 4$. In fact, notice that for $q\ge 4$,
\[
-(\beta_1+\beta_2)(1-\frac{2}{q})+\frac{1}{q}-[-(\beta_1+\beta_2-\frac{1}{2})(1-\frac{2}{q})]
= -\frac{1}{2}+\frac{2}{q} 
\le 0,
\]
which implies 
\[
\eta^{-(\beta_1+\beta_2)(1-\frac{2}{q})+\frac{1}{q}}\ge \eta^{-(\beta_1+\beta_2-\frac{1}{2})(1-\frac{2}{q})},
\]
because $0<\eta<1$. Thus for $q\ge 4$, up to a constant multiple,
\begin{equation}  \label{decoupling constant}
D_{\beta_1,\beta_2,q}=\eta^{-(\beta_1+\beta_2)(\frac{1}{2}-\frac{1}{q})}+\eta^{-(\beta_1+\beta_2)(1-\frac{2}{q})+\frac{1}{q}}.   
\end{equation}  

From (\ref{change of variables 1}), we see that each pair $(k, l)$ corresponds to a vector $\Vec{\xi}=\vec{\xi}(k, l)=
(\xi_1(l,k), \xi_2(l,k), \xi_3(l,k))$
such that $\xi_1(l,k)=\frac{l\sqrt{k}}{L\sqrt{K}} $, $ \xi_2(l,k)=\frac{l}{L} \frac{k/K-1}{2} $ and $\xi_3(l,k)=
 \frac{l}{L} \frac{k/K+1}{2}$.  Define 
\begin{equation}
{\mathcal R}_\sigma := \bigg\{(k, l)\in\mathbb Z^2: k\sim K, l\sim L,  \vec{\xi}(k, l)\in\sigma\bigg\}\,. 
\end{equation}

\begin{lemma}  \label{R-sigma}
 ${\mathcal R}_\sigma$ is contained a rectangle 
\[
I_\sigma \times J_\sigma \subset \{(k,l)\in \mathbb Z^2: k\sim K,l\sim L\}
\]
of dimensions (up to constant multiples)
\[
(1+\eta^{\beta_2}K) \times (1+\eta^{\beta_1}L).
\]
\end{lemma}
\begin{proof} By \eqref{change of variables 1}, a variation of the vector $\Vec{\xi}=(\xi_1,\xi_2,\sqrt{\xi_1^2+\xi_2^2})$ in the null direction corresponds to a change in $\frac{l}{L}$. If $\Vec{\xi}$ moves in the null direction by length $\eta^{\beta_1}$, then $l$ moves by $\eta^{\beta_1}L$ units, so $l$ lies in an interval $J_\sigma$ of size $1+\eta^{\beta_1}L$. 

Similarly,  if $\Vec{\xi}$ moves in the circular direction by length $\eta^{\beta_2}$, then $k$ moves by $\eta^{\beta_2}K$ units, so $k$ lies in an interval $I_\sigma$ of size $1+\eta^{\beta_2}K$.   \end{proof}

By Lemma \ref{R-sigma}, reversing the changes of variables \eqref{change of variables 1} and \eqref{change of variables 2}, we obtain 
\begin{equation}   \label{change back}
\begin{split}
&\Big \|\sum_{\Vec{\xi}\in \sigma}F_{\Vec{\xi}}\,(x_1,x_2,x_3) \Big\|^q_{L^q_{\#}\Big(B(\frac{1}{\eta},KL,KL)\Big)}    
\\ =&\Big\|\sum_{(k,l)\in {\mathcal R}_\sigma} a_{kl}e\Big(lx_1+klx_2+l\sqrt{k}x_3\Big)\Big\|^q_{L^q_{\#}\Big(B(1,1,\frac{1}{\eta L\sqrt{K}})\Big)}.   
\end{split}    
\end{equation}
Henceforth, by \eqref{application of small cap decoupling}, \eqref{decoupling constant} and \eqref{change back}, we can conclude the following lemma. 
\begin{lemma}   \label{after small cap decoupling}
Let $\beta_1\in [\frac{1}{2},1]$ and $\beta_2\in[0,1]$. For $q\ge 4$,
\begin{equation}  \label{variance 3}
G_q\lesssim_\epsilon  \eta^{-\epsilon} \Big[\eta^{-(\beta_1+\beta_2)(\frac{1}{2}-\frac{1}{q})}+\eta^{-(\beta_1+\beta_2)(1-\frac{2}{q})+\frac{1}{q}}\Big] \cdot E_q(\beta_1,\beta_2),
\end{equation}  
where 
\begin{equation}  \label{definition of Eq}
E_q(\beta_1,\beta_2):=\Big(\sum_{\sigma}\Big\|\sum_{(k,l)\in \mathcal R_\sigma} a_{kl}e\Big(lx_1+klx_2+l\sqrt{k}x_3\Big)\Big\|^q_{L^q_{\#}\Big(B(1,1,\frac{1}{\eta L\sqrt{K}})\Big)}\Big)^{\frac{1}{q}}.    
\end{equation}
\end{lemma}

In applications of Lemma \ref{after small cap decoupling}, we need to check that our choices of parameters $\beta_1,\beta_2$ are valid. Most importantly, we need to verify that
\begin{equation}   \label{beta1 da yu 1/2}
 \beta_1 \ge \frac{1}{2}.    
\end{equation}
Moreover, we have to choose the parameters $\beta_1$ and $\beta_2$ to obey the following constraints,  
\begin{equation}  \label{conditions on beta1,2}
\eta^{\beta_1}K\ge 1, \quad  \eta^{\beta_2}L\ge 1.   
\end{equation}
Otherwise the intervals $I_\sigma,J_\sigma$ are of length $1$,  
and the decoupling constant $D_{\beta_1,\beta_2,q}$ gets bigger, which makes the right hand side of \eqref{variance 3} a worse upper bound. From now on, we assume \eqref{conditions on beta1,2} to be true. Of course, we need to check its validity after we make the choices of the parameters $\beta_1,\beta_2$. 

With Lemma \ref{after small cap decoupling} in hand,  it remains to  estimate $E_q(\beta_1,\beta_2)$. We first consider the case when $q$ is an even integer $2n$ with $n\geq 2$. In this case,   because of Lemma \ref{R-sigma} and the fact that the $\mathcal R_\sigma$'s are disjoint, $E_q(\beta_1,\beta_2)^q$ is controlled by  the number of solutions of the following system, called System $(*)$,  containing the following equations and conditions 
from (\ref{condition 1'}) to (\ref{condition 6'}).
\begin{align}
l_1+\dots +l_n &=l_{n+1}+\dots +l_{2n},     \label{condition 1'} 
\\  k_1l_1+\dots +k_n l_n &=k_{n+1}l_{n+1}+\dots +k_{2n} l_{2n}, \label{condition 2'}
\\ l_1\sqrt{k_1} +\dots + l_n\sqrt{k_n}&=l_{n+1}\sqrt{k_{n+1}}+\dots + l_{2n}\sqrt{k_{2n}}+O(\eta L\sqrt{K}),   \label{condition 3'}
\\ k_i \sim K, l_i\sim L, & \quad \forall i=1,...,2n,   \label{condition 4'}
\\ \text{diam}(k_1,\cdots,k_{2n})&\lesssim \eta^{\beta_1}K ,  \label{condition 5'}
\\ \text{diam}(l_1,\cdots,l_{2n})&\lesssim \eta^{\beta_2}L . \label{condition 6'}
\end{align}
Here the notation $\text{diam}(a_1,\dots,a_k)$ was introduced in Section \ref{notations}.  \\


Comparing with the system restricted by \eqref{condition 1}-\eqref{condition 4}, we simply add two more conditions \eqref{condition 5'}, \eqref{condition 6'} in System $(*)$. We call them ``localization conditions", since they localize the variables $k_i,l_i$. \\

For $q=4$, we have the following useful lemma.
\begin{lemma}  \label{Lemma number of solutions q=4}
\begin{equation} \label{number of solutions q=4}
E_4(\beta_1,\beta_2) \lesssim_\epsilon K^{\frac{1}{4}+\epsilon}L^{\frac{1}{4}} \Big( \eta^{2(\beta_1+\beta_2)}K^2 L+\eta^{2\beta_1}K^2+\eta^{2\beta_2}L^2 \Big)^{\frac{1}{4}}.    
\end{equation}
\end{lemma}
Before we start the proof, we show that Lemma \ref{Lemma number of solutions q=4}, in combination with Lemma \ref{after small cap decoupling}, gives us an essentially sharp bound on $G_4$, namely, the following corollary. 
\begin{corollary} \label{G_4}
\[
G_4 \lesssim_\epsilon K^{\frac{1}{2}+\epsilon}L^{\frac{1}{2}}.
\]    
\end{corollary} 
\begin{proof}[Proof of Corollary \ref{G_4}]
We insert \eqref{number of solutions q=4} into the bound on $G_4$ \eqref{variance 3}, and let $\beta_1+\beta_2=1$. Then
\begin{equation}  \label{G4 bound}
G_4\lesssim_\epsilon K^\epsilon\Big[ \eta^{\frac{1}{4}}K^{\frac{3}{4}}L^{\frac{1}{2}}+\eta^{-\frac{1}{4}}  K^{\frac{1}{4}+}L^{\frac{1}{4}}(\eta^{\beta_1}K+\eta^{\beta_2}L)^{\frac{1}{2}} \Big].    
\end{equation}
We choose $\beta_1,\beta_2$ by setting 
\begin{equation}  \label{definition of betas in q=4}
\eta^{\beta_1}K=\eta^{\beta_2}L.    
\end{equation}
The equation \eqref{definition of betas in q=4} determines $\beta_1,\beta_2$, since we also have $\beta_1+\beta_2=1$. We can solve these two equations to get
\begin{equation} \label{solution of beta1}
\eta^{\beta_1}=\eta^{\frac{1}{2}} \Big(\frac{L}{K}\Big)^{\frac{1}{2}}.
\end{equation}
We recall from \eqref{relations between L,K,eta} that
\[
L\le K, 0<\eta<1.
\]
Hence we can conclude from \eqref{solution of beta1} that 
\[
\beta_1 \ge \frac{1}{2}. 
\]
In addition,  it follows from \eqref{solution of beta1} that 
\begin{equation}  \label{check of parameters in q=4}
\eta^{\beta_1}K=\eta^{\beta_2}L= (\eta KL)^\frac{1}{2}\ge 1,    
\end{equation}
which is what we expect in \eqref{conditions on beta1,2}.
Finally, we insert \eqref{check of parameters in q=4} back into \eqref{G4 bound} and get 
\[
G_4\lesssim_\epsilon K^\epsilon ( \eta^{\frac{1}{4}} K^{\frac{3}{4}}L^{\frac{1}{2}}+K^{\frac{1}{2}}L^{\frac{1}{2}}) \lesssim K^{\frac{1}{2}+\epsilon}L^{\frac{1}{2}},
\]
where the last inequality follows from $\eta K\le  1$. 
\end{proof}

We return to the proof of Lemma \ref{Lemma number of solutions q=4}:
\begin{proof} [Proof of Lemma \ref{Lemma number of solutions q=4}]
We follow Bourgain and Watt's idea in \cite[Proposition 7]{BourgainWatt1st}. In this case, $n=\frac{q}{2}=2$, and $E_4(\beta_1,\beta_2)^4$ is bounded by the number of solutions of the System ($*$) constrained by \eqref{condition 1'} to \eqref{condition 6'}. First, we disregard the third inequality \eqref{condition 3'} and focus on the two algebraic equations \eqref{condition 1'}, \eqref{condition 2'}. This move turns out to cost us negligible loss, which shows the power of the localization conditions \eqref{condition 5'}, \eqref{condition 6'}. 

The first equation \eqref{condition 1'} tells us that $l_4$ is determined once we choose $l_1,l_2,l_3$. If we set $\Delta k_i=k_i-k_4$ for $i=1,2,3$, then $|\Delta k_i|\lesssim \eta^{\beta_1}K$ by \eqref{condition 5'}, indicating that 
$\Delta k_i$  varies in a relatively small range.  
There are $\sim K$ many choices of the variable $k_4$. Each $k_i$ with $i=\{1,2,3\}$, is determined uniquely by $k_4$ and $\Delta k_i$. Thus it suffices to consider how many $\Delta k_1, \Delta k_2,\Delta k_3$ we can take when $k_4$ is fixed. 

Subtracting $\eqref{condition 1'}\times k_4$ 
from $\eqref{condition 2'}$, we obtain 
\begin{equation}
 l_1\Delta k_1 + l_2\Delta k_2= l_3\Delta k_3\,,  
\end{equation}
which implies 
\begin{equation}     \label{variance of conditions 1}
 (l_1-l_3)\Delta k_1+(l_2-l_3)\Delta k_2=l_3(\Delta k_3-\Delta k_1-\Delta k_2).
\end{equation}

The main observation we will use is that a certain nonzero integer has only a few factors, namely,
\[
d(v)\lesssim_\epsilon v^\epsilon \quad \text{ for all } v\ge 1,
\]
where $\epsilon>0$ is any small number. If we can set up an equation about integers, and both sides are nonzero, then the number of divisors of one side can not exceed the number of divisors of the other side. For example, in \eqref{variance of conditions 1}, if both sides are nonzero and the left hand side is given, then we do not have many choices of $l_3$ and $\Delta k_3-\Delta k_1-\Delta k_2$. This simple observation will be used repeatedly in the following proof, and it is also useful for dealing with $E_{2n}(\beta_1,\beta_2)^{2n}$, where $n\ge 3$. \\

We examine \eqref{variance of conditions 1} and distinguish several cases: \\

\underline{Case (a)}: Both sides are not zero, which is equivalent to $\Delta k_3 \neq \Delta k_1+\Delta k_2$. For a given vector $(\Delta k_1, \Delta k_2, l_1-l_3,l_2-l_3)$, there are $O(K^\epsilon)$ many choices of  $l_3$ and $\Delta k_3$, because $K\ge L$ and 
$l_3$ and $\Delta k_3-\Delta k_1-\Delta k_2$ are divisors of the given left side of (\ref{variance of conditions 1}).
Taking into account the number of ways to choose $\Delta k_1, \Delta k_2, l_1-l_3,l_2-l_3, k_4$, we conclude that in this case there are
\[
\lesssim_\epsilon K^\epsilon (\eta^{\beta_1}K)^2 (\eta^{\beta_2}L)^2 K=\eta^{2(\beta_1+\beta_2)}K^{3+\epsilon}L^2
\]
different solutions of the System ($*$). \\

\underline{Case (b)}: Both sides of \eqref{variance of conditions 1} are zero, which implies that $\Delta k_3=\Delta k_1+\Delta k_2$. Therefore, $\Delta k_3$ is determined once $\Delta k_1, \Delta k_2$ are given. In this case, \eqref{variance of conditions 1} implies that 
\begin{equation}  \label{variance of conditions 2}
(l_1-l_3)\Delta k_1=(l_3-l_2)\Delta k_2,   
\end{equation}
and we further consider two sub-cases as follows. \\

{\bf Sub-case} (b.i): Both sides of \eqref{variance of conditions 2} are nonzero. If $l_1$, $l_1-l_3,\Delta k_1$ are constant, then  $l_2,l_3,\Delta k_2$ are essentially determined as we did in Case a).  In this sub-case, we can select
\[
\Delta k_1,l_1-l_3, l_1, k_4, 
\]
without any restriction. That leads to
\[
\lesssim_\epsilon \eta^{\beta_1+\beta_2} K^{2+\epsilon}L^2
\]
different solutions of the System ($*$). \\

{\bf Sub-case} (b.ii): Both sides of \eqref{variance of conditions 2} are zero. We consider the following three sub-sub-cases.\\

Sub-sub-case (b.ii.1): $\Delta k_1=l_3-l_2=0$ or $\Delta k_2=l_1-l_3=0$.

By symmetry, the analysis for these two scenarios is the same, and the total number of solutions is 
\[
\lesssim \eta^{\beta_1+\beta_2} K^2 L^2.
\]

Sub-sub-case (b.ii.2): $\Delta k_1=\Delta k_2=0$, there are 
\[
\lesssim \eta^{2\beta_2}KL^3
\]
distinct solutions.

Sub-sub-case (b.ii.3): $l_1-l_3=l_2-l_3=0$. There are 
\[
\lesssim \eta^{2\beta_1}K^3 L
\]
distinct solutions. 

Having taken all possible cases into consideration, we can summarize that, after simplification, there are 
\[
\lesssim_\epsilon K^{1+\epsilon}L \Big( \eta^{2(\beta_1+\beta_2)}K^2 L+\eta^{2\beta_1}K^2+\eta^{2\beta_2}L^2 \Big)  
\]
many solutions of the System ($*$) at $q=4$.
\end{proof}

\quad

At this stage, we have all the tools we need. To prove Proposition \ref{main proposition statment}, we still resort to Lemma \eqref{after small cap decoupling}, but unlike the case when $q=4$, we cannot reduce the estimate of a general $E_q(\beta_1,\beta_2)$ to a counting problem. However, there is a connection between $E_4(\beta_1,\beta_2)$ and $E_q(\beta_1,\beta_2)$.

\begin{lemma}   \label{lemma E4 and Eq}
If \eqref{conditions on beta1,2} holds, then for $q\ge 4$, 
\begin{equation}  \label{E4 and Eq}
E_q (\beta_1,\beta_2) \lesssim (\eta^{\beta_1+\beta_2}KL)^{1-\frac{4}{q}} E_{4}(\beta_1,\beta_2)^\frac{4}{q}.    
\end{equation}
\end{lemma}

\begin{proof}
By Lemma \ref{R-sigma} and \eqref{conditions on beta1,2}, for each $\sigma$,
\[
\Big\|\sum_{(k,l)\in \mathcal R_\sigma}a_{kl}e(lx_1+klx_2+l\sqrt{k}x_3)\Big\|_{L^\infty}
\le  |I_\sigma||J_\sigma| \lesssim \eta^{\beta_1+\beta_2}KL. 
\]
Accordingly,
\begin{equation}   \label{q and 4}
\begin{split}
&\Big\|\sum_{(k,l)\in \mathcal R_\sigma} a_{kl}e\Big(lx_1+klx_2+l\sqrt{k}x_3\Big)\Big\|^q_{L^q_{\#}\Big(B(1,1,\frac{1}{\eta L \sqrt{K}})\Big)} 
\\ \le & \Big\|\sum_{(k,l)\in \mathcal R_\sigma}a_{kl}e(lx_1+klx_2+l\sqrt{k}x_3)\Big\|_{L^\infty}^{q-4}
\\ \times & \Big\|\sum_{(k,l)\in \mathcal R_\sigma} a_{kl}e\Big(lx_1+klx_2+l\sqrt{k}x_3\Big)\Big\|^4_{L^4_{\#}\Big(B(1,1,\frac{1}{\eta L \sqrt{K}})\Big)}
\\ \lesssim & (\eta^{\beta_1+\beta_2}KL)^{q-4} 
\\ \times & \Big\|\sum_{(k,l)\in \mathcal R_\sigma} a_{kl}e\Big(lx_1+klx_2+l\sqrt{k}x_3\Big)\Big\|^4_{L^4_{\#}\Big(B(1,1,\frac{1}{\eta L \sqrt{K}})\Big)}.
\end{split}    
\end{equation}
Thence, by comparing the definition of $E_q(\beta_1,\beta_2)$, \eqref{definition of Eq} and \eqref{q and 4}, we find that
\begin{equation}  \label{E44 and Eqq}
E_q(\beta_1,\beta_2)^q\lesssim (\eta^{\beta_1+\beta_2}KL)^{q-4} E_4(\beta_1,\beta_2)^4.   \end{equation}
The inequality \eqref{E4 and Eq} follows by taking the $\frac{1}{q}$-th power on both sides of \eqref{E44 and Eqq}.
\end{proof}

Now comes the proof of our main Proposition.

\begin{proof}  [Proof of Proposition \ref{main proposition statment}]
We assume the validity of \eqref{conditions on beta1,2} and will verify it later. By Lemma \ref{Lemma number of solutions q=4} and Lemma \ref{lemma E4 and Eq}, $E_q(\beta_1,\beta_2)$ can be bounded as   
\begin{equation}  \label{upper bound on Eq}
\begin{split}
E_q(\beta_1,\beta_2)\lesssim_\epsilon K^\epsilon &(\eta^{\beta_1+\beta_2}KL)^{1-\frac{4}{q}} (KL)^{\frac{1}{q}} 
\\ \times & (\eta^{2(\beta_1+\beta_2)}K^2L +\eta^{2\beta_1}K^2+\eta^{2\beta_2}L^2)^{\frac{1}{q}}. 
\end{split}
\end{equation}
We insert \eqref{upper bound on Eq} into \eqref{variance 3}, and deduce that
\begin{equation}   \label{q>4 before optimization}
\begin{split}
G_q\lesssim_\epsilon  \eta^{-\epsilon} &\Big[\eta^{-(\beta_1+\beta_2)(\frac{1}{2}-\frac{1}{q})}+\eta^{-(\beta_1+\beta_2)(1-\frac{2}{q})+\frac{1}{q}}\Big] (\eta^{\beta_1+\beta_2}KL)^{1-\frac{4}{q}} 
\\  & \Big[  KL \Big(\eta^{2(\beta_1+\beta_2)}K^2L+\eta^{2\beta_1}K^2+\eta^{2\beta_2}L^2\Big)  \Big]^{\frac{1}{q}}.  
\end{split}
\end{equation}
Notice that in \eqref{q>4 before optimization}, we hide $K^\epsilon$ in $\eta^{-\epsilon}$ since $\frac{1}{\eta}\ge K$. Next we choose the parameters $\beta_1,\beta_2$ in order to optimize the upper bound in \eqref{q>4 before optimization}. For simplification of notation, if we let $\beta_1+\beta_2=\beta$, \eqref{q>4 before optimization} becomes 
\begin{equation}   \label{idk what to label}
\begin{split}
G_q \lesssim_\epsilon  \eta^{-\epsilon} & \Big[\eta^{-\beta(\frac{1}{2}-\frac{1}{q})}+\eta^{-\beta(1-\frac{2}{q})+\frac{1}{q}}\Big] (\eta^{\beta}KL)^{1-\frac{4}{q}}   
\\ &\Big[KL(\eta^{2\beta}K^2L+\eta^{2\beta_1}K^2+\eta^{2\beta_2}L^2)\Big]^{\frac{1}{q}}.
\end{split}    
\end{equation}
If we set
\begin{equation}  \label{choices of beta1 beta2 q}
\eta^{\beta_1}K=\eta^{\beta_2}L=(\eta^\beta KL)^{\frac{1}{2}},   
\end{equation}
then \eqref{idk what to label} takes the form 
\[
G_q \lesssim_\epsilon \eta^{-\epsilon} (KL)^{1-\frac{2}{q}} \Big[  \eta^{\beta(\frac{1}{2}-\frac{2}{q})}+ \eta^{\frac{1}{q}(1-\beta)}  \Big](1+\eta^\beta K)^{\frac{1}{q}}.
\]
Finally, if we set 
\begin{equation}   \label{q>4 beta value}
\beta=\frac{2}{q-2} \le 1,
\end{equation}
we have 
\[
G_q \lesssim_\epsilon \eta^{-\epsilon} \eta^{\frac{q-4}{q(q-2)}} (KL)^{1-\frac{2}{q}}(1+\eta^{\frac{2}{q-2}}K)^{\frac{1}{q}}. 
\]
This is \eqref{main proposition}. 

It remains to verify the conditions \eqref{beta1 da yu 1/2} and \eqref{conditions on beta1,2}. On one hand, we do not know what $q$ to choose at this step, so the values of $\beta$ and $\beta_1$ are unknown, and that is why we leave \eqref{beta1 da yu 1/2} as an assumption. By \eqref{choices of beta1 beta2 q} and the fact that $0<\eta<1$, it is easy to see that \eqref{beta1 da yu 1/2} is equivalent to the assumption \eqref{main proposition assumption} in Proposition \ref{main proposition statment}. 
On the other hand, by \eqref{relations between L,K,eta}, 
\[
1\le \frac{1}{\eta} \le KL,
\]
and since $\beta \le 1$, we find that
\[
\eta^\beta KL \ge \eta KL \ge 1.
\]
Therefore
\[
\eta^{\beta_1}K=\eta^{\beta_2}L=(\eta^\beta KL)^{\frac{1}{2}}\ge 1,
\]
which is exactly \eqref{conditions on beta1,2}. 

We have completed our proof of the main Proposition \ref{main proposition statment}.\\
\end{proof}

\section{Bounds on certain exponential sum using the Bombieri-Iwaniec Method}\label{bounds on exponential sum}

In this section, we derive an effective upper bound for a certain type of double exponential sum, which appears in the study of Gauss's Circle Problem and Dirichlet's Divisor Problem. \\

\subsection{Main theorem on bounds of the exponential sum}
Let us begin with some definitions and assumptions. Let $\epsilon>0$ and $C_1,C_2,\dots,C_5\ge 2$ be real constants. Let $F(x)$ be a real function that is three times continuously differentiable for $1\le x\le 2$, and let $g(x),G(x)$ be functions of bounded variation on the interval $[1,2]$. $M$ and $T$ denote large positive parameters and $H\ge 1$.  We set  
\begin{equation}  \label{definition of S}
S :=\sum_{H\le h\le 2H} g\Big(\frac{h}{H}\Big) \sum_{M\le m\le 2M} G\Big(\frac{m}{M}\Big)e\Big(\frac{hT}{M}F\Big(\frac{m}{M}\Big)\Big),
\end{equation}
which is the standard form of the exponential sum encountered in the study of both the Circle and Divisor Problems. \\

Suppose moreover that, on the interval $[1,2]$, the derivatives $F^{(1)}(x)$, \\ $F^{(2)}(x)$,$F^{(3)}(x)$ satisfy:
\begin{equation}  \label{condition on F 1}
C_r\ge |F^{(r)}(x)|\ge C_{r}^{-1}  \quad (r=1,2,3)    
\end{equation}
and  
\begin{equation}  \label{condition on F 2}
|F^{(1)}(x)F^{(3)}(x)-3F^{(2)}(x)^2|\ge C_{4}^{-1},     
\end{equation}
for some constant $C_4$.

We focus on the following two cases. 

\begin{definition}  [Case (A)] \label{case A definition}
Let $H,M,T$ satisfy the three conditions: 
\begin{equation}  \label{case A}
\begin{cases}
H\ge M^{-9}T^{4}(\log T)^{\frac{171}{140}}  \quad  &\text{ if }M<T^{-\frac{7}{16}},
\\ H\ge M^{11}T^{-6}(\log T)^{\frac{171}{140}}  \quad &\text{ if }M>T^{\frac{9}{16}}, 
\\ H \le MT^{-\frac{49}{164}}.    &
\end{cases}    
\end{equation}    
\end{definition}  

\begin{definition}  [Case (B)]  \label{case B definition}
Let $H,M,T$ satisfy the two conditions:
\begin{equation}   \label{case B}
\begin{cases}
M \le  C_5 T^{\frac{1}{2}},  &
\\ H \le  \min\{M^{\frac{35}{69}}T^{-\frac{2}{23}}, B_0 M^{\frac{3}{2}}T^{-\frac{1}{2}}\}, &
\end{cases}    
\end{equation}
in which $B_0$ is a positive constant depending on $C_1,...,C_5$.   
\end{definition}
If we are in one of these two cases, then Huxley's result \cite{Huxley03} on the second spacing problem can be applied. That leads to the following lemma:  
\begin{lemma}   \label{main lemma-111} 
If we are in Case (A) or Case (B), and we suppose that 
\begin{equation}  \label{condition 1---'}
N^{6-q}\gg H^{2q-6} \Big(\frac{M^3}{T} \Big)^{4-q},    
\end{equation}
then
\begin{equation}   \label{upper bound middle form}
\begin{split}
S \lesssim_\epsilon T^\epsilon & \frac{M^{\frac{5}{2}}}{T^{\frac{1}{2}}} \Big(\frac{H^2 T}{M^3}\Big)^{\frac{11}{17q}} \Big(\frac{TH}{M^3}\Big)^{1-\frac{2}{q}-\frac{q-4}{q(q-2)}}  
\\ \times &   N^{\frac{1}{2}-\frac{57}{17q}-\frac{2(q-4)}{q(q-2)}}\Big(1+\Big(\frac{M^3}{HT}\Big)^{\frac{2}{q-2}}\frac{T^{\frac{1}{2}}}{M^{\frac{3}{2}}} N^{\frac{3}{2}-\frac{4}{q-2}}\Big)^{\frac{1}{q}}, 
\end{split}
\end{equation}
where $N$ is defined by
\begin{equation}  \label{definition of N}
N\sim  \begin{cases}
H (\frac{M}{H})^{\frac{41}{25}} T^{-\frac{49}{100}} (\log T)^{\frac{969}{14000}} \quad   & \text{ in case (A)}, 
\\ \min\Big\{ \frac{M^{\frac{7}{8}}(\log T)^{\frac{969}{5600}}}{T^{\frac{3}{20}}H^{\frac{29}{40}}}, \frac{M^2}{H^{\frac{1}{3}}T^{\frac{2}{3}}}  \Big\}  \quad & \text{ in case (B)}.
\end{cases}    
\end{equation} 
\end{lemma}

\vspace{0.5cm}

By the definition of $N$ in (\ref{definition of N}), in Case (A), (\ref{condition 1---'}) becomes 
\begin{equation}  \label{condition 1-------}
H^{\frac{2q-6}{6-q}+\frac{16}{25}}M^{\frac{34}{25}}\ll T^{\frac{51}{100}}(\log T)^{\frac{969}{14000}}  \,.  
\end{equation}
The inequality (\ref{upper bound middle form}) can be simplified as 
\begin{equation}  \label{upper bound final form}
\begin{split}
\frac{S}{H} \lesssim_\epsilon T^\epsilon & \Big(\frac{H}{M} \Big)^{-\frac{8}{25}+\frac{36}{25q}+\frac{7(q-4)}{25q(q-2)}} T^{\frac{51}{200}+\frac{29}{100q}-\frac{q-4}{50q(q-2)}}
\\ \times & \Big(1+\Big(\frac{H}{M} \Big)^{\frac{14}{25(q-2)}-\frac{24}{25}} T^{-\frac{1}{25(q-2)}-\frac{47}{200}} \Big)^{\frac{1}{q}} ,
\end{split}   
\end{equation}
where we hide all $\log T$ powers in $T^\epsilon$.\\

\begin{theorem} \label{main theorem}
\quad 
\begin{itemize}
\item  In Case (A),   (\ref{condition 1-------}) yields  (\ref{upper bound final form}). 
\item  In Case (B), if (\ref{range of q}) and 
\begin{equation}   \label{condition for reduction from B to A}
M^{-\frac{27}{23}}T^{\frac{53}{92}}<H<M^{-9}T^4 (\log T)^{\frac{171}{140}}  \,
\end{equation}
hold, then (\ref{condition 1-------}) implies  (\ref{upper bound final form}) as well. 
\end{itemize}
\end{theorem}

\begin{remark}
The key point here is that once \eqref{range of q} and \eqref{condition for reduction from B to A} hold, we have a unified version of the condition \eqref{condition 1-------} and the estimate for $S$ \eqref{upper bound final form} in both Case (A) and Case (B).
\end{remark}

\subsection{Bourgain-Watt's argument}
A detailed account of the Bombieri-Iwaniec method using more Analysis language was given in \cite[Sections 7-12]{BourgainWatt1st}. 
Here we follow the reasoning in \cite[Section 5]{BourgainWatt2nd} and depart from their argument when we invoke the estimate on the first spacing problem. 

Let $H,T,M,N$ be as above in \eqref{definition of S}, \eqref{definition of N}, and let $q$ be the same parameter as we used in Section \ref{first spacing problem}. The other parameters are defined by 
\begin{equation}  \label{definitions of parameters}
\begin{split}
R& \sim \Big(\frac{M^3}{NT}\Big)^{\frac{1}{2}} ,
\\ L & \sim \frac{HQ}{R^2} \ge 1 ,
\\ K & \sim \frac{NQ}{R^2} \ge 1 ,
\\ \eta & \sim \Big(\frac{Q}{R}\Big)^2 (KL)^{-1} \sim \frac{R^2}{NH},
\end{split}    
\end{equation}
and $Q$ lies in the range
\begin{equation}  \label{definition of Q}
 R \le Q\le 3H \le \frac{3}{64C_2} N.    
\end{equation}
The inequality \eqref{definition of Q} also gives us relations between the parameters $R,H,N$. From (\ref{definitions of parameters}),  we see that the parameters $\eta, K$ and $L$ satisfy 
\begin{equation}  \label{L,K,eta}
L\le K\le \frac{1}{\eta}\le KL.    
\end{equation}  \\

Let us sketch Bourgain and Watt's argument in \cite[Section 5]{BourgainWatt2nd}, without presenting some technical details which are illustrated clearly and nicely in \cite[Section 5]{BourgainWatt2nd}. 
Typically, in the study of both the Circle and Divisor Problems,  one encounters the exponential sum \cite[(5.1)]{BourgainWatt2nd}
\[
S=\sum_{H\le h \le H_1} \sum_{M\le m\le M_2}e \Big(\frac{hT}{M}F\Big( \frac{m}{M} \Big) \Big),
\]
where $1\le \frac{H_1}{H}\le 2$ and $1\le \frac{M_1}{M}\le 2$.  The range of $m$ is divided  into intervals $I_j$ of length $N$, so that the long sum over $m$ is divided into short sums $S_j$ over the shorter interval $I_j$. On each $I_j$, the Taylor expansion
can be applied to replace the function $F$ by a quadratic polynomial in $m$.  Via a use of Dirichlet's approximation theorem, the linear coefficient is approximated
by a rational number $\frac{a}{r}$. $I_j$ is uniquely determined by the rational number $a/r$ so that one can also denote 
by $I_{a/r}$. Group those intervals $I_{{a'}/{r'}}$ for which $r'\ge H$ with a ``nearby" interval $I_{\frac{a}{r}}$ for which $r\le \frac{R^2}{H}$.  This process makes the interval  $I_{a/r}$ into a longer one, still denoted by ${I_{a/r}}$, called a major arc. Those $I_{a/r}$ with $\frac{R^2}{H}\le r\le H$ are called minor arcs. The range $\frac{R^2}{H}\le r\le H$ can be shrunk
to  $R\le r\le H$ by Huxley's method.  The treatments for major arcs are relatively easy and the contributions of the double exponential sum from them can be controlled by (first term in \cite[(5.6)]{BourgainWatt2nd})
\[
\frac{MR\log H}{\sqrt{HN}},
\]
which is negligible,  compared with the contributions from the minor arcs. For the original classifications of major and minor arcs, see \cite[Section 3]{HuxleyCircle1}. 
The minor arcs $I_{a/r}$'s with $R\leq r\leq H$ are more difficult to examine.  By the pigeonhole principle, one can assume 
the denominator $r$ lies between $Q$ and $2Q$ for some dyadic number $Q\in [R, H]$, while the numerator 
$a\sim A$ for some dyadic number $A\leq Q$. After applying the Poisson summation formula, we find that the double exponential sum over minor arcs is bounded by \cite[(5.6) and (5.9)]{BourgainWatt2nd}
\[
(\log H)^2 \Big[|\mathcal{C}(A,Q)|\frac{Q}{R}\sqrt{HN}\log N
+\frac{R^2}{Q} \sum_{I_{\frac{a}{r}}\in \mathcal{C}(A,Q)}\Big| \mathop{\sum}\limits_{\substack{L\le l\le 2L \\K\le k\le 2K}} e\Big( \Vec{x}_{\frac{a}{r}}\cdot \Vec{y}_{(k,l)} \Big) \Big|\Big],
\]
where the first term arises from the error terms in the Poisson summation. 
The principal contribution is from the second term, where 
\[
\Vec{y}_{(k,l)}=\Big(k,lk,l\sqrt{k},\frac{l}{\sqrt{k}}\Big)\in \mathbb R^4, 
\]
and $\Vec{x}_{\frac{a}{r}} \in \mathbb R^4$ is a vector depending on $\frac{a}{r}$ that will be specified at the end of this section, when we discuss the second spacing problem.
$\mathcal{C}(A,Q)$ is a subset of $\{\frac{a}{r}:a\sim A, r\sim Q\}$, and it also can be interpreted as a subset of those corresponding intervals or minor arcs. $A=0$ is allowed and $\mathcal{C}(0,Q)$ refers to those intervals 
$I_{\frac{a}{r}}$ which Huxley calls ``bad intervals". The precise definition of ``bad intervals" is given in \cite[Section 2]{HuxleyZeta5} and \cite[Section 2]{Huxley03}.
The purpose of separating out such intervals will be clear if one follows Huxley's work on the second spacing problem.  All $\mathcal{C}(A,Q)$'s  form a partition of minor arcs with $Q\le r\le 2Q$.  \\

For the triple sum in the principal contribution,
\[
\sum_{I_{\frac{a}{r}}\in \mathcal{C}(A,Q)}\Big| \mathop{\sum}\limits_{\substack{L\le l\le 2L \\K\le k\le 2K}} e\Big( \Vec{x}_{\frac{a}{r}}\cdot \Vec{y}_{(k,l)} \Big) \Big|, 
\]
one can bound it by 
\[
|\mathcal{C}(A,Q)|^{1-\frac{2}{q}} (R^{-8}H^4N^2Q^2 V B(A,Q;V))^{\frac{1}{q}}G_q,  
\]
which is \cite[(5.12)]{BourgainWatt2nd}, employing the double large sieve inequality presented in \cite[Sections 5]{BourgainWatt1st}. Here $G_q$, the mean value of certain exponential sum, is defined as in \eqref{original form}, and it is equal to $\sqrt[q]{A_q}$ in \cite{BourgainWatt2nd}; $V$ is a parameter; and $B(A,Q,;V)$ is the quantity appearing in the second spacing problem. $V$ will be chosen to minimize the product $V B(A,Q;V)$. After invoking Huxley's result on the second spacing problem \cite{Huxley03} and discussing many cases, Bourgain and Watt concluded the following lemma. 
\begin{lemma}  \cite[(5.22)]{BourgainWatt2nd}   \label{BW 5.22}
For each $\epsilon>0$, 
\[
S\lesssim_\epsilon T^\epsilon \max_{R\le Q\lesssim Q_2}  \Big(\frac{R}{Q} \Big)^{3-\frac{6}{q}} \Big(\frac{MR}{N} \Big) \Big(\frac{H}{R}\Big)^{\frac{22}{17q}} G_q, 
\]    
where \[
Q_2= R\Big(\frac{H}{R}\Big)^{39/119}\Big(\log (2H/R)\Big)^{-\frac{3}{4}}.
\]
\end{lemma}
We remind that readers may refer to  \eqref{definition of S}, \eqref{definition of N} and \eqref{definitions of parameters} for the definitions of parameters.  

\subsection{Proof of Theorem \ref{main theorem}}

We now use Lemma \ref{BW 5.22} to prove our Lemma \ref{main lemma-111} and Theorem \ref{main theorem} in this subsection.

\begin{proof} [Proof of Lemma \ref{main lemma-111}]
By Lemma \ref{BW 5.22} and the upper bound of $G_q$ in Proposition \ref{main proposition statment}, if we know that
\begin{equation}    \label{condition 1---}
\frac{R^2}{NH} \gg \Big(\frac{H}{N}\Big)^{\frac{q-2}{q-4}},
\end{equation}
then  
\begin{equation}   \label{upper bound after first spacing problem improvement}
\begin{split}
S \lesssim_\epsilon \max_{R\le Q\lesssim Q_2}  T^\epsilon   & R^{3-\frac{6}{q}} \Big(\frac{MR}{N} \Big) \Big(\frac{H}{R}\Big)^{\frac{22}{17q}} \eta^{\frac{q-4}{q(q-2)}} \Big(\frac{NH}{R^4}\Big)^{1-\frac{2}{q}}   
\\ \times & Q^{\frac{2}{q}-1}\Big(1+\eta^{\frac{2}{q(q-2)}}\Big(\frac{N}{R^2}\Big)^{\frac{1}{q}}Q^{\frac{1}{q}}\Big).
\end{split}    
\end{equation}
We keep in mind that the size of $\eta$ is independent of $Q$ (as shown in the last line of \eqref{definitions of parameters}), and so $Q$ does not appear in the first line of \eqref{upper bound after first spacing problem improvement}. The second line of \eqref{upper bound after first spacing problem improvement} can be written as 
\begin{equation}   \label{second line}
Q^{\frac{2}{q}-1}\Big(1+\eta^{\frac{2}{q(q-2)}}\Big(\frac{N}{R^2}\Big)^{\frac{1}{q}}Q^{\frac{1}{q}}\Big)=Q^{\frac{2}{q}-1}+\eta^{\frac{2}{q(q-2)}}\Big(\frac{N}{R^2}\Big)^{\frac{1}{q}}Q^{\frac{3}{q}-1}.   
\end{equation}
Since $q\ge 4$, both terms in \eqref{second line} decrease with respect to $Q$. Therefore the maximum of the right-hand side of \eqref{upper bound after first spacing problem improvement} is attained at $Q=R$. We let $Q=R$ in \eqref{upper bound after first spacing problem improvement} and obtain 
\begin{equation}    \label{plug in Q=R}
S \lesssim_\epsilon T^\epsilon  \Big(\frac{MR}{N}\Big) \Big(\frac{H}{R} \Big)^{\frac{22}{17q}} \Big(\frac{NH}{R^2}\Big)^{1-\frac{2}{q}-\frac{q-4}{q(q-2)}}\Big(1+\Big(\frac{R^2}{NH}\Big)^{\frac{2}{q-2}}\frac{N}{R}\Big)^{\frac{1}{q}}.   
\end{equation}
Next we insert the definition \eqref{definitions of parameters} of $R$  into \eqref{plug in Q=R}, and the bound becomes 
\begin{equation}   \label{22}
\begin{split}
S \lesssim_\epsilon T^\epsilon & \frac{M^{\frac{5}{2}}}{T^{\frac{1}{2}}} \Big(\frac{H^2 T}{M^3}\Big)^{\frac{11}{17q}} \Big(\frac{TH}{M^3}\Big)^{1-\frac{2}{q}-\frac{q-4}{q(q-2)}}  
\\ \times &   N^{\frac{1}{2}-\frac{57}{17q}-\frac{2(q-4)}{q(q-2)}}\Big(1+\Big(\frac{M^3}{HT}\Big)^{\frac{2}{q-2}}\frac{T^{\frac{1}{2}}}{M^{\frac{3}{2}}} N^{\frac{3}{2}-\frac{4}{q-2}}\Big)^{\frac{1}{q}}. 
\end{split}    
\end{equation}
This is exactly \eqref{upper bound middle form}. \\

In addition, when we examine the condition \eqref{condition 1---}, we notice that it is equivalent to the condition \eqref{condition 1---'} in Lemma \ref{main lemma-111}:
\begin{equation}  \label{condition of N repeated}
N^{6-q}\gg H^{2q-6} \Big(\frac{M^3}{T}\Big)^{4-q}.    
\end{equation}
Therefore, we complete the proof of Lemma \ref{main lemma-111}.
\end{proof}

\begin{proof}[Proof Theorem \ref{main theorem}]

Of course, we can insert the definition \eqref{definition of N} of $N$ into \eqref{22} and obtain estimates for $S$ in Case (A) and Case (B) respectively. However, it would be better if we have a unified version of estimates that work for both cases. With this goal in mind, we notice that the terms in the first line of \eqref{22} are independent of $N$. In the second line, the two exponents of $N$ satisfy the inequalities 
\[
\begin{split}
\frac{1}{2}-\frac{57}{17q}-\frac{2(q-4)}{q(q-2)} &<0 , 
\\ \frac{3}{2}-\frac{4}{q-2} &<0, 
\end{split}
\]
for $4\le q\le 4.5$, so the right hand side of \eqref{22} decreases with respect to $N$. If we replace $N$ by a smaller number, the upper bound would still hold.   \\

On the other hand,  by \eqref{range of q}, $6-q>0$.  If we replace $N$ by a smaller number and the condition \eqref{condition of N repeated} still holds, then it also holds for the original choice of $N$. 

In Case (A), if we insert the definition \eqref{definition of N} of $N$ in Case (A) into \eqref{condition 1---'} and \eqref{upper bound middle form}, then we obtain \eqref{condition 1-------} and \eqref{upper bound final form} after simplification.

In Case (B), we first define
\[
N_A=H (\frac{M}{H})^{\frac{41}{25}} T^{-\frac{49}{100}} (\log T)^{\frac{969}{14000}},
\]
and
\[
N_B=\min\Big\{ \frac{M^{\frac{7}{8}}(\log T)^{\frac{969}{5600}}}{T^{\frac{3}{20}}H^{\frac{29}{40}}}, \frac{M^2}{H^{\frac{1}{3}}T^{\frac{2}{3}}}  \Big\}. 
\]
Suppose that we have \eqref{condition for reduction from B to A}, and therefore $N_B>N_A$. If \eqref{condition 1-------} holds, \eqref{condition 1---'} is valid for $N=N_A$. By the above discussion, \eqref{condition 1---'} is also valid for $N_B$. Thus we have the upper bound \eqref{upper bound middle form} for $S$ with $N=N_B$. Again by the above discussion, assuming the desired range of \eqref{range of q}, \eqref{upper bound middle form} must be valid for $N=N_A$, which is \eqref{upper bound final form}. 
 We have reduced the condition \eqref{condition 1---'} and the estimate \eqref{upper bound middle form} in Case (B) to those in Case (A). Therefore, (\ref{condition 1-------}) implies (\ref{upper bound final form}) under the hypotheses in Theorem \ref{main theorem}.

\end{proof}
\begin{remark}  \label{issue}
We remind readers that the conditions \eqref{definition of Q} and \eqref{L,K,eta} cannot be taken for granted from their definitions. Fortunately, Bourgain and Watt have verified them in \cite[Section 6]{BourgainWatt2nd}. There is a subtle issue that was carefully explained in \cite[Section 6]{BourgainWatt2nd}. Nevertheless, we stress it one more time: the condition
\begin{equation}  \label{crucial condition}
R\le H,  
\end{equation}
is not always satisfied. The importance of \eqref{crucial condition} comes from the fact that it is a necessary condition to guarantee that
\[
L \ge 1, \quad \eta K \le 1, 
\]
which are assumed to be true throughout our argument. But if we contemplate on this issue, we realize $L\le 1$ implies that a certain exponential sum has only one term in it, so we do not need heavy machinery to estimate $S$ in this case. Instead, we can obtain a nice upper bound by using elementary methods from the beginning. This obstacle was first overcome by Huxley in \cite[Page 377]{Huxley1996}  (and also in \cite[(3.26), (3.27)]{Huxley03}). In this paper, our conclusion is that \eqref{plug in Q=R} still holds even if \eqref{crucial condition} fails. So the results in Theorem \ref{main theorem} are valid no matter if \eqref{crucial condition} holds or not.

Moreover, Huxley's argument will appear in Section \ref{final argument}, in order to exclude certain undesirable cases. They are instances when \eqref{crucial condition} fails, but this type of ``equivalence" is not obvious without careful computations. It turns out that, in such cases, elementary methods give us better estimates than \eqref{upper bound middle form} does.\\
\end{remark}

\subsection{The second spacing problem}
At the end of Section \ref{bounds on exponential sum}, we discuss the second spacing problem. Given a minor arc $I_{\frac{a}{r}}$, where $\frac{a}{r}$ is a reduced fraction, we first define an important parameter $m=m_{\frac{a}{r}}$ determined by $\frac{a}{r}$ in the following way.  Let $m$ be the  nearest integer to 
\[
\Big(\frac{T}{M^2}F'\Big(\frac{m}{M}\Big) \Big)^{-1} \Big(\frac{a}{r}\Big),
\]
where $-1$ in the superscript denotes the inverse function. We define the vector $\Vec{x}_{\frac{a}{r}}$ by
\begin{equation} \label{x sub a/r}
\Vec{x}_{\frac{a}{r}}:=\Big( \frac{\overline{a}}{r}, \frac{\overline{a}c}{r}, \frac{1}{\sqrt{\mu r^3}}, \frac{\kappa}{\sqrt{\mu r^3}} \Big). 
\end{equation}
Here the parameters are given precisely by,
\[
\begin{array}{rcl}
a\overline{a} & \equiv & 1(\text{ mod }r) ,  \\
\mu & = & \frac{1}{2} \frac{T}{M^3}F^{(2)}\Big( \frac{m}{M}\Big) , \\ 
\nu &= &\frac{\frac{T}{M^2}F'\Big(\frac{m}{M}\Big)-\frac{a}{r}}{2\mu}, \quad \,\,\,\, \, \qquad \text{ where } |\nu|\le 1 , \\   
c &=& \lfloor r \frac{T}{M}F\Big(\frac{m}{M}\Big) -\mu \nu^2 \rfloor,  \\ 
\kappa &=& \Big\{r \frac{T}{M}F\Big(\frac{m}{M}\Big) -\mu \nu^2  \Big\}, \quad  \text{ where } 0\le \kappa <1. 
\end{array}
\]
In this first line, we only need the information of $\overline{a}$ in congruence classes modulo $r$. In the last line above, $\{\cdot \}$ denotes the fractional part.  

The second spacing problem asks for the number of pairs $\Big(\frac{a}{r}, \frac{a_1}{r_1}\Big)$ with $a,a_1\sim A$, $r,r_1\sim Q$ such that 
\begin{align}
\Big\|\frac{\overline{a}}{r}-\frac{\overline{a_1}}{r_1} \Big\| & \lesssim \frac{1}{KL},    
\\ \Big\|\frac{\overline{a}c}{r}-\frac{\overline{a_1}c_1}{r_1} \Big\| & \lesssim \frac{1}{L},    
\\ \Big| \frac{1}{\sqrt{\mu r^3}}-\frac{1}{\sqrt{\mu_1 r_1^3}} \Big| & \lesssim \frac{1}{L\sqrt{K}},  
\\ \Big| \frac{\kappa}{\sqrt{\mu r^3}}-\frac{\kappa_1}{\sqrt{\mu_1 r_1^3}} \Big| & \lesssim \frac{\sqrt{K}}{L}. 
\end{align}
They can be further simplified in the form 
\begin{align}
\Big\|\frac{\overline{a}}{r}-\frac{\overline{a_1}}{r_1} \Big\| & \le \Delta_1,  \quad \text{where }\Delta_1 \ll 1,  \label{con1}
\\ \Big\|\frac{\overline{a}c}{r}-\frac{\overline{a_1}c_1}{r_1} \Big\| & \le \Delta_2,   \label{con2}
\\ \Big| \frac{\mu_1 r_1^3}{\mu r^3} -1 \Big| & \le \Delta_3,   \label{con3}
\\ | \kappa-\kappa_1 | & \le \Delta_4,  \label{con4}
\end{align}
but we do not specify $\Delta_1,\Delta_2,\Delta_3,\Delta_4$ here. The inequalities (\eqref{con1}-\eqref{con4}) are important in studying the pointwise estimates of $\Big|\zeta(\frac{1}{2}+it)\Big|$, the Circle and Divisor Problems. They first appear in the work of Bombieri and Iwaniec \cite{BombieriIwaniec}, where the authors only utilized two conditions \eqref{con1} and \eqref{con3}. Their novel idea is that every pair $\Big(\frac{a}{r}, \frac{a_1}{r_1}\Big)$ determines a unique matrix
\[
\mathcal{M}=\begin{pmatrix}
\alpha & \beta \\
\gamma & \delta
\end{pmatrix} \in SL_2(\mathbb Z)  \,,
\]
in the following way
\begin{equation}  \label{mtx equa}
\begin{pmatrix}
 a_1 \\ r_1    
\end{pmatrix}=\begin{pmatrix}
\alpha & \beta \\
\gamma & \delta
\end{pmatrix} \begin{pmatrix}
 a \\ r    
\end{pmatrix},   
\end{equation}
where $\alpha\delta-\beta\gamma=1$ and   $-\frac{1}{2}r r_1 < \gamma \le \frac{1}{2}rr_1$.  Moreover, (\ref{con1})
implies that $|\gamma|\leq \Delta_1rr_1$. In other words, (\ref{con1}) is used to decrease the bound of $|\gamma|$.\\

Conversely, if we know $\frac{a}{r}$ and the matrix
\begin{equation}   \label{matrix M conditions}
\mathcal{M}=\begin{pmatrix}
\alpha & \beta \\
\gamma & \delta
\end{pmatrix}, \quad \text{ where }\alpha\delta-\beta\gamma=1, \text{ and }  |\gamma|\le \Delta_1 rr_1,    
\end{equation}
then $\frac{a_1}{r_1}$ is determined by (\ref{mtx equa}). Henceforth,  to count the number of pairs, we may fix a matrix $\mathcal{M}$ first and count the number of $\frac{a}{r}$ such that  the pair derived through \eqref{mtx equa} satisfies \eqref{con1} to \eqref{con4}. 

Originally, we know that $\frac{a}{r} \sim \frac{T}{M^2}\ge 1$ lies in a relatively long interval. Using a simple method and  only the conditions \eqref{con1}, \eqref{con3}, Huxley and Watt \cite[Section 2]{HuxleyZeta1} showed that if we fix a matrix $\mathcal{M}$, then we are able to force $\frac{a}{r}$ into a much shorter interval. This restricts the number of $\frac{a}{r}$ we can choose. Summing over the number of possible $\frac{a}{r}$ first and then over all matrices $\mathcal{M}$ satisfying \eqref{matrix M conditions}, Huxley and Watt were able to recover the work of Bombieri and Iwaniec on the second spacing problem.

Later in \cite{HuxleyZeta4}, Huxley made a clever use of all four conditions \eqref{con1}-\eqref{con4}, connecting the original problem with counting the number of lattice points sitting close to a curve. Using the resonance curve method \cite{Huxley1996}, \cite{HuxleyRC}, he was able to further improve the bounds on the second spacing problem in \cite{HuxleyZeta5} and \cite{Huxley03}.

\section{ New estimate for Gauss's Circle Problem and Dirichlet's Divisor Problem }  \label{final argument}

Using the standard hyperbola method and the partial summation formula, one can show that \cite[Theorem 4.5 and 4.8]{GrahamKolesnik1991} 
\begin{equation}  \label{divisor error term}
\Delta (X)=-2\sum_{m\le \sqrt{X}} \psi\Big(\frac{X}{m}\Big)+O(1),    
\end{equation}
and 
\begin{equation}   \label{circle error term}
\begin{split}
R(X)=-4 \Big[ &\sum_{m\le \sqrt{X}}\psi\Big(\frac{X}{4m+1}\Big)-\sum_{m\le \sqrt{X}}\psi\Big(\frac{X}{4m-1}\Big)
\\ +&\sum_{m\le \sqrt{X}}\psi\Big( \frac{X}{4m}-\frac{1}{4}\Big)-\sum_{m\le \sqrt{X}}\psi\Big( \frac{X}{4m}-\frac{1}{4}\Big) \Big] +O(1),
\end{split}    
\end{equation}
where $\psi(t)=(t-\lfloor t\rfloor)-\frac{1}{2}$ is the sawtooth function. From now, we will replace the variable $X$ by $T$, in order to be consistent with our notations in Section \ref{bounds on exponential sum}. There is a well-known truncated Fourier expansion of the sawtooth function, namely, 
\begin{equation}  \label{st func expansion}
\psi(t)=\text{Im} \sum_{1\le h\le Y}  \frac{e(ht)}{\pi h}+O\Big(\frac{1}{1+\|t\|Y}\Big),     
\end{equation}
where we take $Y=MT^{-\theta^*}$. If we insert \eqref{st func expansion} into \eqref{divisor error term} and \eqref{circle error term}, and divide the range of $h,m$ into dyadic intervals $h\asymp H$, $m\asymp M$, we encounter exponential sums of the form \eqref{definition of S}: 
\begin{equation}   \label{definition of S again}
S=\sum_{h\asymp H} \sum_{m\asymp M} e\Big(\frac{hT}{M}F\Big(\frac{m}{M}\Big)\Big),    
\end{equation}
where $1\le H\le MT^{-\theta^*}$, $1\le M\le T^{\frac{1}{2}}$, and $F$ takes the following forms:
\[
F(z)=\frac{1}{z}, \, \frac{1}{z+\frac{1}{4}}, \,  \frac{1}{z-\frac{1}{4}}, \, \frac{1}{4z}-\frac{M}{4T}, \, \frac{1}{4z}+\frac{M}{4T}. 
\]
We see that all the above $F$'s are smooth on $[1,2]$ and satisfy the two conditions \eqref{condition on F 1}, \eqref{condition on F 2} stated at the beginning of Section \ref{bounds on exponential sum}. For bounds on the error term in \eqref{st func expansion} and other reductions, we refer the readers to \cite[Section 7]{BourgainWatt2nd}. In order to prove the desired bounds on the Circle and Divisor Problems in Theorem \ref{theorem in introduction}, it is enough to show that (\cite[(7.7)]{BourgainWatt2nd})
\begin{equation}  \label{goal}
\frac{S}{H}\lesssim_\epsilon T^{\theta^*+\epsilon}.  
\end{equation}
We can also put some restrictions on $H,M$:
\begin{equation}   \label{restriction on M}
M  \le T^{\frac{1}{2}},    
\end{equation}
and 
\begin{equation}  \label{restriction on H ''} 
T^{\frac{7\theta^*-2}{2}} \le H\le  MT^{-\theta^*}.
\end{equation}
Here \eqref{restriction on H ''}  is explained in \cite[(7.6)]{BourgainWatt2nd}. Before we prove \eqref{goal}, we further restrict the range of $H$, and this was mentioned at the end of Remark \ref{issue}.

With Kusmin-Landau's inequality (\cite[Thm 2.1]{GrahamKolesnik1991}) and van der Corput's inequality (\cite[Thm 2.2]{GrahamKolesnik1991}) applied to the single exponential sum over $m$ in \eqref{definition of S}, we have (\cite[(6.9)]{BourgainWatt2nd})
\begin{equation}  \label{simple case}
S \lesssim H\Big(\Big(\frac{HT}{M^2}\Big)^{-1}+ M\Big(\frac{HT}{M^3}\Big)^{\frac{1}{2}}\Big) \sim \frac{H^{\frac{3}{2}}T^{\frac{1}{2}}}{M^{\frac{1}{2}}},    
\end{equation}
where the last step is due to $M \le T^{\frac{1}{2}}\le (HT)^{\frac{3}{5}}$. It follows from \eqref{simple case} that 
\[
\frac{S}{H} \lesssim \Big(\frac{HT}{M} \Big)^{\frac{1}{2}}.
\]
If we know that 
\[
\frac{H}{M}\le T^{2\theta^*-1},
\]
then \eqref{goal} is satisfied. Therefore it remains to consider the case 
\begin{equation}  \label{another lower bound on H}
H > MT^{2\theta^*-1},   
\end{equation}
where $2\theta^*-1> -\frac{3}{8}$. If we combine \eqref{restriction on H ''} with \eqref{another lower bound on H}, then $H$ lies in the range 
\begin{equation}   \label{restriction on H}
\Big[\max(T^{\frac{7\theta^*-2}{2}},MT^{2\theta^*-1}),  MT^{-\theta^*}\Big].     
\end{equation}
\begin{remark}
We apply Kusmin-Landau's inequality when the product of the length of the interval and the second derivative is $\lesssim 1$, since this condition implies that the first derivative has a small perturbation. We apply van der Corput's inequality when the product is $\gtrsim 1$. 
\end{remark} 

In the following argument, we always assume that $T$ is sufficiently large. In other words, $T$ is larger than some absolute constant. Otherwise, \eqref{goal} becomes trivial. 

Of course, to prove \eqref{goal}, we would like to apply Theorem \ref{main theorem}. But before that, we have to make sure that the prerequisites in Theorem \ref{main theorem} are satisfied. Namely, we are in either Case (A) (Definition \ref{case A definition}) or Case (B) (Definition \ref{case B definition}). In addition, it would be better if we have \eqref{condition for reduction from B to A}. We need to verify the validity of those conditions. 

By \eqref{restriction on M}, 
\[
M\le T^{\frac{1}{2}} < T^{\frac{9}{16}}, 
\]
then by \eqref{restriction on H} and \eqref{definition of theta}, 
\[
H\le MT^{-\theta^*}< MT^{-0.3144}<MT^{-\frac{49}{164}}.
\]
The second and third conditions in \eqref{case A} always hold. We are in Case (A) if and only if the first condition in \eqref{case A}: 
\begin{equation}  \label{1st condition in case A}
H\ge M^{-9}T^4 (\log T)^{\frac{171}{140}}  \quad \text{ if } M<T^{\frac{7}{16}},     
\end{equation}
holds. According to this, we consider two mutually exclusive cases:

Case I:
\begin{equation}  \label{case I}
H\ge M^{-9}T^4 (\log T)^{\frac{171}{140}}.
\end{equation}

Case II: \eqref{case I} fails. 

Obviously, if we are in Case I, then \eqref{1st condition in case A} holds, and therefore we are in Case A. 
In the following lemma, we show that if we are in Case II, then we are in Case B, and what is more, \eqref{condition for reduction from B to A} holds. In this way, we know that the estimate \eqref{upper bound final form} in Theorem \ref{main theorem} can be applied. 
\begin{lemma}  \label{case B reduction}
Suppose that $T$ is sufficiently large. If \eqref{case I} fails, then the two conditions in Case B \eqref{case B} are satisfied. Moreover, \eqref{condition for reduction from B to A} is true. 
\end{lemma}
\begin{proof}
The first condition in \eqref{case B} holds if we choose $C_5=3$. To prove the second inequality in \eqref{case B}, we show that
\begin{equation}  \label{case B 1st condition}
H\le M^{\frac{35}{69}}T^{-\frac{2}{23}}    
\end{equation}
and 
\begin{equation}   \label{case B 2nd condition}
H\le B_0 M^{\frac{3}{2}}T^{-\frac{1}{2}}    
\end{equation}
are valid, where $B_0$ is some absolute constant.  

On one hand, by the upper bound of $H$ in \eqref{restriction on H} and the negation of \eqref{case I}, we have 
\begin{equation}  \label{H 11}
\begin{split}
H\le & \Big(\frac{M}{T^{-\theta^*}}\Big)^{\frac{328}{345}} \Big(\frac{T^4(\log T)^{\frac{171}{140}}}{M^9}\Big)^{1-\frac{328}{345}}
\\ =& M^{\frac{35}{69}} T^{\frac{68-328\theta^*}{345}}(\log T)^{\frac{969}{16100}}. 
\end{split}    
\end{equation}
By the definition \eqref{definition of theta} of $\theta^*$, we know that $\theta^* \ge 0.3144$, so 
\begin{equation}  \label{theta computation 1}
\frac{68-328\theta^*}{345}<-0.1018 <-\frac{2}{23}.    
\end{equation}
As long as $T$ is large enough, a negative power of $T$ would dominate $(\log T)^{\frac{969}{16100}}$. It follows from \eqref{H 11} and \eqref{theta computation 1} that \eqref{case B 1st condition} holds. \\

On the other hand, the upper bound on $H$ in \eqref{restriction on H} implies that
\[
M \ge HT^{\theta^*},
\]
and by the first lower bound of $H$ in \eqref{restriction on H},
\[
\Big(\frac{M^3}{T}\Big)^{\frac{1}{2}} \ge  \Big(\frac{H^3 T^{3\theta^*}}{T}\Big)^{\frac{1}{2}}    
=H H^{\frac{1}{2}} T^{\frac{3\theta^*-1}{2}} 
 \ge H T^{\frac{13\theta^*-4}{4}}
 \ge  H T^{\frac{3}{400}+\frac{13}{4000}},
\]
where we use $\theta^* \ge 0.311$ in the last line. If $T$ is large in terms of $B_0$,  \eqref{case B 2nd condition} holds. 

\quad 

Next we show that \eqref{condition for reduction from B to A} is satisfied if \eqref{case I} fails. The upper bound in \eqref{condition for reduction from B to A} follows from the negation of \eqref{case I}. We turn to show the lower bound. The upper bound of $H$ in \eqref{restriction on H} indicates that 
\[
M\ge HT^{\theta^*}.
\]
Then it is easily seen that 
\[
H^{23}M^{27}\ge H^{50}T^{27\theta^*},
\]
and we invoke the first lower bound of $H$ in \eqref{restriction on H} to achieve 
\begin{equation}  \label{33}
H^{23}M^{27}\ge T^{25(7\theta^*-2)+27\theta^*}= T^{202\theta^*-50}>T^{\frac{53}{4}},    
\end{equation}
where we use $\theta^*>0.3144$ in the last inequality. Now we can conclude from \eqref{33} that 
\[
H > M^{-\frac{27}{23}}T^{\frac{53}{92}},
\]
which is the lower bound in \eqref{condition for reduction from B to A}. Thus, Lemma \ref{case B reduction} has been proved. 
\end{proof}

If we integrate the above discussions of Case I and Case II (i.e. Lemma \ref{case B reduction}) with Theorem \ref{main theorem}, we arrive at the following proposition,
\begin{proposition}   \label{last proposition used}
If  \eqref{range of q} and \eqref{condition 1-------} are valid, then we can use \eqref{upper bound final form} to estimate the double exponential sum $S$ \eqref{definition of S again} in the Circle and Divisor Problems.    \\
\end{proposition}

We now turn to the proof of \eqref{goal}. As we will see, our optimal choice of $q$ depends on the relations between $H,M$ and $T$, and there is a technical issue when $H$ is very small, namely, $H< MT^{-\frac{3}{8}}$. The reason is that $q<4$ in this case. However, we do not have to worry about this case, since we have $H\ge MT^{2\theta^*-1}>MT^{-\frac{3}{8}}$ in \eqref{restriction on H}.  

Now we let $x\in \mathbb R$ satisfy  
\begin{equation}  \label{definition of x}
H=MT^{-x}.    
\end{equation}
By \eqref{restriction on H}, we know that
\begin{equation}  \label{range of x}
-\frac{3}{8}<x\le -\theta^* \le -\frac{49}{164}.   
\end{equation} 
\begin{definition}[Choice of $q$]
We let
\begin{equation}  \label{choice of q}
q=q_x=\frac{2}{5\sqrt{\frac{-1-8x}{2(1-14x)}}-1}+2.   
\end{equation}   
\end{definition}
It is easy to verify that \eqref{range of q} holds. In fact, $q$ is an increasing function of $x$ for $x$ in the range \eqref{range of x}, and 
\[
q_{-\frac{3}{8}}=4, \quad  q_{-\theta^*}\approx 4.29<4.35.
\]
By Proposition \ref{last proposition used}, if \eqref{condition 1-------} is satisfied, then we have 
\begin{equation}  \label{upper bound final form repeated}
\begin{split}
\frac{S}{H} \lesssim_\epsilon T^\epsilon & \Big(\frac{H}{M} \Big)^{-\frac{8}{25}+\frac{36}{25q}+\frac{7(q-4)}{25q(q-2)}} T^{\frac{51}{200}+\frac{29}{100q}-\frac{q-4}{50q(q-2)}}
\\ \times & \Big(1+\Big(\frac{H}{M} \Big)^{\frac{14}{25(q-2)}-\frac{24}{25}} T^{-\frac{1}{25(q-2)}-\frac{47}{200}} \Big)^{\frac{1}{q}}.
\end{split}      
\end{equation}
Inserting the definition of $x$ \eqref{definition of x} into \eqref{upper bound final form repeated}, we deduce that
\begin{equation}  \label{44}
\begin{split}
\frac{S}{H} \lesssim_\epsilon &  T^{-\frac{8}{25}x+\frac{51}{200}+\frac{36x}{25q}+\frac{29}{100q}+\frac{14x-1}{50}\cdot \frac{q-4}{q(q-2)}+\epsilon}   
\\ \times & \Big( 1+T^{(\frac{7x}{25}-\frac{1}{50})\frac{2}{q-2}-\frac{24x}{25}-\frac{47}{200}} \Big)^{\frac{1}{q}}.
\end{split}    
\end{equation}
Also, \eqref{condition 1-------} can be deduced from the inequality
\begin{equation}  \label{check 1}
\frac{\frac{7}{25}x-\frac{1}{50}}{\frac{41}{25}x+\frac{49}{100}}< \frac{q-2}{q-4}=\frac{1}{2-5\sqrt{\frac{-1-8x}{2(1-14x)}}} .  \end{equation}
Again, here we use the fact that a positive power of $T$ is greater than a positive power of $\log T$ if $T$ is sufficiently large. We remind the readers that both the numerator and denominator of  the fraction on the far left-hand side of \eqref{check 1} are negative.

To verify the condition \eqref{condition 1-------} and one more inequality \eqref{check 2} which helps simplify further computations, we need the following lemma, whose proof is postponed to the end of this section:
\begin{lemma}  \label{checking}
For $x$ in the range \eqref{range of x}, \eqref{check 1} holds  and 
\begin{equation}   \label{check 2}
\frac{\frac{24}{25}x+\frac{47}{200}}{\frac{7}{25}x-\frac{1}{50}}\le \frac{2}{q-2}= 5\sqrt{\frac{-1-8x}{2(1-14x)}}-1 .
\end{equation} 
\end{lemma}
This lemma has a direct corollary. 
\begin{corollary}   \label{final corollary}
For $-\frac{3}{8}\le x\le -\theta^*$,   
\[
\frac{S}{H} \lesssim_\epsilon  T^{-\frac{8}{25}x+\frac{51}{200}+\frac{36x}{25q}+\frac{29}{100q}+\frac{14x-1}{50}\cdot\frac{q-4}{q(q-2)}+\epsilon} . 
\]
\end{corollary}
\begin{proof}
By Lemma \ref{checking}, the condition \eqref{check 1} is checked, and therefore the condition \eqref{condition 1-------} is true. The inequality \eqref{check 2} implies that
\[
T^{(\frac{7x}{25}-\frac{1}{50})\frac{2}{q-2}-\frac{24x}{25}-\frac{47}{200}} \lesssim 1,
\]
for $T$ large (note that $\frac{7x}{25}-\frac{1}{50}<0$). So the inequality in Corollary \ref{final corollary} follows from Proposition \ref{last proposition used} and \eqref{44}. 
\end{proof}
Now we can embark on the proof of \eqref{goal}.
\begin{proof}  [Proof of \eqref{goal}]
By the definition \eqref{choice of q} of $q$ and from simple algebra,  
\[
\begin{split}
\frac{q-4}{q-2}=1-\frac{2}{q-2}=2-5\sqrt{\frac{-1-8x}{2(1-14x)}}\,.  
\end{split}
\]
Since $-1-8x>0$ and $1-14x>0$, the exponent of $T$ involving $q$ in Corollary \ref{final corollary} can be written as 
\begin{equation}  \label{algebra 1}
\begin{split}
&\frac{36x}{25q}+\frac{29}{100q}+\frac{14x-1}{50}\cdot\frac{1}{q}\cdot \frac{q-4}{q-2} 
\\ =& \Big[\frac{36x}{25}+\frac{29}{100}+\frac{14x-1}{50}\Big(2-5\sqrt{\frac{-1-8x}{2(1-14x)}}\Big)\Big]\frac{1}{q}
\\ =& \Big[ \frac{36x}{25}+\frac{29}{100}+\frac{14x-1}{25} +\frac{1-14x}{10}\sqrt{\frac{-1-8x}{2(1-14x)}}\Big]\frac{1}{q}
\\ =& \frac{1}{q}\Big(\frac{8x+1}{4}+\frac{1}{10}\sqrt{\frac{(1-14x)(-1-8x)}{2}}\Big)
\\ =& \frac{1}{q} \cdot \frac{\sqrt{-1-8x}}{20} \Big( \sqrt{2(1-14x)} -5\sqrt{-1-8x}\Big). 
\end{split}    
\end{equation}
By the definition \eqref{choice of q} of $q$, we also have
\begin{equation}  \label{algebra 2}
\frac{q}{2}= \frac{5\sqrt{-1-8x}}{5\sqrt{-1-8x}-\sqrt{2(1-14x)}}.    
\end{equation}
Substituting \eqref{algebra 2} into \eqref{algebra 1}, we obtain
\begin{equation}  \label{algebra 3}
\begin{split}
&\frac{36x}{25q}+\frac{29}{100q}+\frac{14x-1}{50}\cdot\frac{1}{q}\cdot \frac{q-4}{q-2}
\\ =& -\frac{1}{200} (\sqrt{2(1-14x)}-5\sqrt{-1-14x})^2.
\end{split}   
\end{equation}
The equality \eqref{algebra 3}, together with Corollary \ref{final corollary}, leads to 
\begin{equation}  \label{upper bound useful 3}
\frac{S}{H}\lesssim_\epsilon T^{-\frac{8}{25}x-\frac{1}{200}\Big(\sqrt{2(1-14x)}-5\sqrt{-1-8x}\Big)^2+\frac{51}{200}+\epsilon}.
\end{equation}
It can be quickly checked (and also observed from the graph in Definition \ref{theta}) that the function
\[
-\frac{8}{25}x-\frac{1}{200}\Big(\sqrt{2(1-14x)}-5\sqrt{-1-8x}\Big)^2+\frac{51}{200}
\]
is increasing with respect to $x$ on $[-\frac{3}{8},-\theta^*]$, and is equal to $\theta^*$ at $x=-\theta^*$ by the definition \eqref{definition of theta} of $\theta$. This finishes the proof of \eqref{goal}. 

\end{proof}

Lastly, we give a proof for Lemma \ref{checking}.
\begin{proof} [Proof of Lemma \ref{checking}] Throughout this proof, $-\frac{3}{8}\le x\le -\theta^*$. The inequality
\eqref{check 1} can be written as
\[
\frac{\frac{7}{25}x-\frac{1}{50}}{\frac{41}{25}x+\frac{49}{100}} <1+\frac{2}{q-4},
\]
which is equivalent to 
\begin{equation}  \label{55}
\frac{\frac{7}{25}x-\frac{1}{50}}{\frac{41}{25}x+\frac{49}{100}}-1< \frac{2}{q-4}.    
\end{equation}
After simplification, \eqref{55} becomes 
\begin{equation}   \label{66}
-17\frac{8x+3}{164x+49}<\frac{2}{q-4}.    
\end{equation}
The left-hand side of \eqref{66} is increasing with respect to $x$. $q=q_x\ge 4$ is also an increasing function of $x$, so the right-hand side of \eqref{66} is decreasing with respect to $x$. To prove \eqref{66}, we only need to verify it at $x=-\theta^*$, where $\theta^*$ was defined in \eqref{definition of theta}. It is easy to do so using a calculator. Thus we see that \eqref{check 1} holds.

\quad

To prove \eqref{check 2}, we notice that this inequality can be rewritten as 
\begin{equation}  \label{77}
\frac{248x+43}{56x-4}=\frac{\frac{24}{25}x+\frac{47}{200}}{\frac{7}{25}x-\frac{1}{50}}+1 \le 5\sqrt{\frac{-1-8x}{2(1-14x)}}.    
\end{equation}
Since both sides of \eqref{77} are positive, we can square it. Therefore \eqref{77} is equivalent to 
\begin{equation}  \label{88}
\Big(\frac{248x+43}{56x-4}\Big)^2 \le 25 \frac{-1-8x}{2(1-14x)}. \end{equation}
We clear denominators on both sides of \eqref{88} and divide both sides by $2$. Thus, 
\begin{equation}  \label{9}
(248x+43)^2 (1-14x)\le 200 (-1-8x)(1-14x)^2.    
\end{equation}
Since $1-14x>0$ for $x\le -\theta^*$, we cancel the common factor $1-14x$ on both sides of \eqref{9} to derive an alternative inequality
\[
(248x+43)^2 \le 200 (-1-8x)(1-14x), 
\]
which, after some simplification, is equivalent to 
\[
(8x+3)(4888x+683)\le 0.
\]
It is easily seen that the last inequality holds for $-\frac{3}{8}\le x\le -\theta^*$. 
Hence (\ref{check 2}) holds. We thus complete the proof of Lemma \ref{checking}.
\end{proof}

\quad

\end{document}